\documentclass[11pt]{amsart}

\usepackage{amsmath,amsthm,amsfonts,amssymb,amscd,amsbsy,mathtools,dsfont} 
\usepackage{microtype}
\usepackage{enumerate}
\usepackage{tikz}
\usetikzlibrary{decorations.pathreplacing}
\usetikzlibrary{calc}
\usepackage{mathdots}
\usepackage{lscape}
\usepackage{hyperref}
\hypersetup{
    pdftoolbar=true,
    pdfmenubar=true,
    pdffitwindow=false,
    pdfstartview={FitH},
    pdftitle={},
    pdfauthor={},
    pdfsubject={},
    pdfkeywords={}
    pdfnewwindow=true,
    colorlinks=true, 
    linkcolor=blue,
    citecolor=blue,
    urlcolor=black,
}
\usepackage{cleveref}
\usepackage[margin=1.25in]{geometry}

\def\BigColSep{\setlength{\arraycolsep}{9.5pt}}
\newcommand{\G}{\mathsf G}
\newcommand{\HH}{\mathds H}
\newcommand{\CC}{\mathds C}
\newcommand{\ZZ}{\mathds Z}
\newcommand{\RR}{\mathds R}
\newcommand{\dd}{\mathrm d} 
\newcommand{\g}{\mathrm g} 
\newcommand{\SO}{\mathsf{SO}}
\newcommand{\conv}{\operatorname{conv}}
\newtheorem{theorem}{Theorem}
\newtheorem{lemma}[theorem]{Lemma}
\newtheorem{proposition}[theorem]{Proposition}

\newtheorem{mainthm}{\sc Theorem}
\newtheorem{maincor}[mainthm]{\sc Corollary}

\newtheorem*{fincon}{\sc Finiteness Conjecture}

\theoremstyle{definition}

\theoremstyle{remark}
\newtheorem{example}[theorem]{Example}
\newtheorem{remark}[theorem]{Remark}
\newtheorem{question}[theorem]{Question}

\allowdisplaybreaks
\numberwithin{equation}{section}
\numberwithin{theorem}{section}

\newcommand{\drawgrid}{
  \draw[black] (0,0) rectangle ({\gridwidth*\gridcellsize},{\gridheight*\gridcellsize});
  \foreach \x in {1,...,\numexpr\gridwidth-1} {
    \draw[black] ({\x*\gridcellsize},0) -- ({\x*\gridcellsize},{\gridheight*\gridcellsize});
  }
  \foreach \y in {1,...,\numexpr\gridheight-1} {
    \draw[black] (0,{\y*\gridcellsize}) -- ({\gridwidth*\gridcellsize},{\y*\gridcellsize});
  }
  \node at (0,0) [circle,fill=red,inner sep=1pt]{};
  \node at (\gridwidth,\gridheight) [circle,fill=red,inner sep=1pt]{};
}

\newcommand{\gridcellsize}{1} 
\newcommand{\gridwidth}{2}    
\newcommand{\gridheight}{2}   

\title{Counting Homogeneous Einstein metrics}


\author[R. G. Bettiol]{Renato G. Bettiol}
\address{\!\!\!\begin{tabular}{lll}
CUNY Lehman College & & CUNY Graduate Center \\
Department of Mathematics & & Department of Mathematics \\
250 Bedford Park~Blvd W & & 365 Fifth Avenue \\
Bronx, NY, 10468, USA & & New York, NY, 10016, USA
\end{tabular}
}
\email{r.bettiol@lehman.cuny.edu}

\author[H. Friedman]{Hannah Friedman}
\address{University of California, Berkeley \newline
\indent Department of Mathematics \newline
\indent 970 Evans Hall \newline
\indent Berkeley, CA, 94720, USA}
\email{hannahfriedman@berkeley.edu}

\begin{document}

\date{\today}

\begin{abstract}
We present an explicit upper bound on the number of isolated homogeneous Einstein metrics on compact homogeneous spaces whose isotropy representations consist of pairwise inequivalent irreducibles.
This is the BKK bound of the corresponding system of Laurent polynomials and is found combinatorially by computing the volume of a polytope. 
Inspired by a connection with algebraic statistics, we describe this system's BKK discriminant in terms of the principal $A$-determinant of scalar curvature. As a consequence, we confirm the Finiteness Conjecture of B\"ohm--Wang--Ziller in special cases. In particular, we give a unified proof that it holds on all generalized Wallach spaces. Finally, using numerical algebraic geometry, we compute $\sf G$-invariant Einstein metrics on low-dimensional full flag manifolds $\sf G/ T$, where $\G$ is a compact simple Lie group and $\sf T$ is a maximal torus.
\end{abstract}

\maketitle

\section{Introduction}

The problem of finding homogeneous Einstein metrics on a compact homogeneous space is, in essence, an algebraic problem, but one of significant geometric interest~\cite{jablonski-survey,wang-survey}.
In this~paper, we advance the understanding of this problem using ideas from algebraic geometry and combinatorics.
This paper is written for a broad audience, including researchers in differential geometry, enumerative combinatorics, and metric algebraic geometry~\cite{BKS}, so, for the readers' convenience, we recall basic notions from these fields.

A Riemannian metric $\g$ on a manifold $M$ is \emph{Einstein} if its Ricci curvature satisfies
\begin{equation}\label{eq:einstein}
  \operatorname{Ric}_{\g} = \lambda \, \g
\end{equation}
for some constant $\lambda\in \RR$, called its \emph{Einstein constant}.
Constructing Einstein metrics is a difficult problem, and a central question in geometric analysis; see \cite{besse} for a comprehensive introduction.
The Einstein equation \eqref{eq:einstein} is a second-order nonlinear PDE on $M$, but, under symmetry assumptions, it can be reduced to an algebraic equation.
Namely, if a Lie group $\G$ acts transitively on $M$, then a $\G$-invariant metric $\g$ and its Ricci curvature $\operatorname{Ric}_{\g}$ are uniquely determined by their value at any point $p_0\in M$.
In this situation, \eqref{eq:einstein} reduces to a system of Laurent polynomial equations in the entries of $\g_{p_0}$. Positive-definite solutions to this system are in bijective correspondence with $\G$-invariant Einstein metrics on the homogeneous space $M=\G/\mathsf H$, where $\mathsf H$ is the stabilizer of $p_0$. These  are called \emph{homogeneous Einstein metrics}.

The sign of $\lambda$ determines an important trichotomy for homogeneous Einstein manifolds $(M,\g)$. If $\lambda<0$, then $(M,\g)$ is isometric to an Einstein solvmanifold and diffeomorphic to Euclidean space, by the recent proof of the Alekseevskii conjecture~\cite{boehm-lafuente}. If $\lambda=0$, then $(M,\g)$ is flat, and hence isometric to the product of a torus and a Euclidean space; see~\cite[Thm.~7.61]{besse}. If $\lambda>0$, then $(M,\g)$ is compact with finite fundamental group and $\G$ can be assumed compact and semisimple; see \cite[\S 1-2]{jablonski-survey} and \cite[\S 1-3]{wang-survey}. In this paper, we only work with the latter case, and, up to homotheties, we henceforth fix $\lambda=1$.

There are two different but intertwined approaches to studying homogeneous Einstein metrics on compact homogeneous spaces: one is variational, the other is algebraic. First, the variational approach is built on Hilbert’s characterization of Einstein metrics
as critical points of the total scalar curvature functional on unit-volume metrics.
Since the space of $\G$-invariant unit-volume metrics on $\sf G/H$ is finite-dimensional, 
one may study this problem via classical critical point theory, e.g., Morse theory, applied to the scalar curvature function. This has been a fruitful perspective, with foundational contributions by Jensen~\cite{jensen}, Wang--Ziller~\cite{wang-ziller-inv}, and B\"ohm--Wang--Ziller~\cite{bwz}.
Second, the algebraic approach is to directly analyze the corresponding system of Laurent polynomials, which are the Euler--Lagrange equations of the aforementioned variational problem. 
This approach and its interplay with representation theory were used by Wang--Ziller~\cite{wang-ziller-inv} to produce examples of compact simply-connected homogeneous spaces $\G/\mathsf H$ that admit no $\G$-invariant Einstein metrics 
and to classify normal homogeneous Einstein metrics when $\G$ is simple \cite{wang-ziller-ens}. Most  subsequent progress focused on special cases, with notable works by Graev~\cite{graev1,graev2,graev3}, using some of the same tools employed in this paper, and several other papers applying Gr\"obner basis techniques to compute solutions; see Arvanitoyeorgos~\cite{arva-survey} for a survey. 

To write \eqref{eq:einstein} on a compact homogeneous space $\sf G/H$ as a system of Laurent polynomial equations, let $Q$ be a bi-invariant metric on $\G$ and $\mathfrak m=\mathfrak m_1 \oplus\dots\oplus\mathfrak m_\ell$ be a decomposition of the $Q$-orthogonal complement of $\mathfrak h\subset\mathfrak g$ into irreducible $\sf H$-representations. Suppose   the $\mathfrak m_i$ are pairwise inequivalent, so every $\G$-invariant homogeneous metric $\g$ on $\sf G/H$ is \emph{diagonal}, i.e., given by $x_1\,Q|_{\mathfrak m_1}+\dots +x_\ell\,Q|_{\mathfrak m_\ell}$ 
for some ${\bf x}=(x_i)\in\RR^\ell_+$. 
Then $\operatorname{Ric}_\g=\g$ 
if and only if
\begin{equation}\label{eq:system}
\phantom{ \quad \quad  1 \leq i \leq \ell }
  f^\ell_i({\bf x})  \coloneqq \frac{b_i}{2x_i} - \frac{1}{4 d_i} \sum_{j,k = 1}^\ell L_{ijk} \frac{2x_k^2 - x_i^2}{x_ix_jx_k} - 1 = 0, \quad \quad  1 \leq i \leq \ell,
\end{equation}
where $d_i=\dim \mathfrak m_i$, the constants $b_i$ depend on the Cartan--Killing form of $\mathfrak g$, and $L=(L_{ijk})$ is a symmetric tensor of structure constants; see Section~\ref{sec:diffgeom} for details.
Note that \eqref{eq:system} is a system of 
$\ell$ Laurent polynomials in $\ell$ variables, with a total of $2\ell + \binom{\ell + 2}{3}$ nonnegative parameters which we label ${\bf b} = (b_i)$, ${\bf d} = (d_i)$, and $L = (L_{ijk})$.
Our first main result is:

\begin{mainthm}\label{thm:delannoy}
  For a fixed $\ell$ and any parameters ${\bf b}, {\bf d}$, and $L$, the number of isolated solutions to \eqref{eq:system} in $(\mathds C^*)^\ell$, counted with multiplicity, is bounded above by the central Delannoy number
  \begin{equation*}
    D_{\ell-1} = \sum\limits_{k=0}^{\ell-1} 2^k \, \binom{\ell-1}{k}^2.
  \end{equation*}
Thus, on a compact homogeneous space $\G/\mathsf H$ whose isotropy representation consists of $\ell$ pairwise inequivalent irreducible summands, there are at most $D_{\ell-1}$ isolated $\G$-invariant Einstein metrics with $\lambda=1$.
\end{mainthm}

For any integer $\ell \ge 2$, the central Delannoy number $D_{\ell-1}$ counts how many polygonal paths join the opposite corners $(0,0)$ and $(\ell-1,\ell-1)$ of a square grid using only \emph{right}, \emph{up}, and \emph{diagonal} steps.
For example, for $\ell=3$, there are $D_2=13$ such paths on a $2\times 2$ grid:
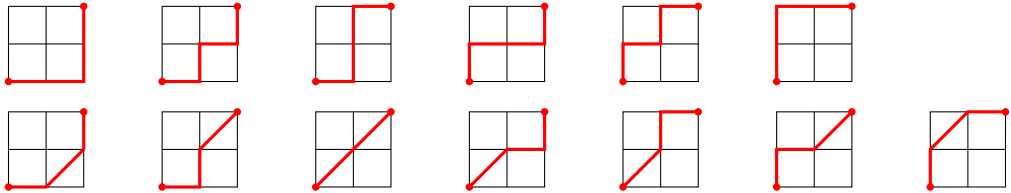
\begin{align*}
\begin{tikzpicture}[scale=0.5]
  \drawgrid
  \draw[red, very thick]
    (0,0) -- ++(1,0) -- ++(1,0) -- ++(0,1) -- ++(0,1);
\end{tikzpicture}
&&
\begin{tikzpicture}[scale=0.5]
  \drawgrid
  \draw[red, very thick]
    (0,0) -- ++(1,0) -- ++(0,1) -- ++(1,0) -- ++(0,1);
\end{tikzpicture}
&&
\begin{tikzpicture}[scale=0.5]
  \drawgrid
  \draw[red, very thick]
    (0,0) -- ++(1,0) -- ++(0,1) -- ++(0,1) -- ++(1,0);
\end{tikzpicture}
&&
\begin{tikzpicture}[scale=0.5]
  \drawgrid
  \draw[red, very thick]
    (0,0) -- ++(0,1) -- ++(1,0) -- ++(1,0) -- ++(0,1);
\end{tikzpicture}
&&
\begin{tikzpicture}[scale=0.5]
  \drawgrid
  \draw[red, very thick]
    (0,0) -- ++(0,1) -- ++(1,0) -- ++(0,1) -- ++(1,0);
\end{tikzpicture}
&&
\begin{tikzpicture}[scale=0.5]
  \drawgrid
  \draw[red, very thick]
    (0,0) -- ++(0,1) -- ++(0,1) -- ++(1,0) -- ++(1,0);
\end{tikzpicture}
&&
\\
\begin{tikzpicture}[scale=0.5]
  \drawgrid
  \draw[red, very thick]
    (0,0) -- ++(1,0) -- ++(1,1) -- ++(0,1);
\end{tikzpicture}
&&
\begin{tikzpicture}[scale=0.5]
  \drawgrid
  \draw[red, very thick]
    (0,0) -- ++(1,0) -- ++(0,1) -- ++(1,1);
\end{tikzpicture}
&&
\begin{tikzpicture}[scale=0.5]
  \drawgrid
  \draw[red, very thick]
    (0,0) -- ++(1,1) -- ++(1,1);
\end{tikzpicture}
&&
\begin{tikzpicture}[scale=0.5]
  \drawgrid
  \draw[red, very thick]
    (0,0) -- ++(1,1) -- ++(1,0) -- ++(0,1);
\end{tikzpicture}
&&
\begin{tikzpicture}[scale=0.5]
  \drawgrid
  \draw[red, very thick]
    (0,0) -- ++(1,1) -- ++(0,1) -- ++(1,0);
\end{tikzpicture}
&&
\begin{tikzpicture}[scale=0.5]
  \drawgrid
  \draw[red, very thick]
    (0,0) -- ++(0,1) -- ++(1,0) -- ++(1,1);
\end{tikzpicture}
&&
\begin{tikzpicture}[scale=0.5]
  \drawgrid
  \draw[red, very thick]
    (0,0) -- ++(0,1) -- ++(1,1) -- ++(1,0);
\end{tikzpicture}
\end{align*}
The first few values in the sequence of central Delannoy numbers (see \cite{oeis}) are the following:
\begin{align*}
&D_1=3,& \; &D_2= 13,& \; &D_3= 63,& \; &D_4= 321,& \; &D_5= 1\,683,& \\
&D_6= 8\,989,& \; &D_7= 48\,639,& \; &D_8= 265\,729,& \; &D_9= 1\,462\,563,& \; &D_{10} = 8\,097\,453,& \dots  
\end{align*}

The bound in Theorem~\ref{thm:delannoy} is the so-called \emph{Bernstein-Khovanskii-Kushnirenko (BKK)~bound} for the system \eqref{eq:system}, given by \emph{Bernstein's Theorem} (see \Cref{thm:bkk}), which states that the mixed volume of the Newton polytopes of such a system bounds the number of isolated complex solutions. 
One expects that the number of $\G$-invariant Einstein metrics on a homogeneous space $\sf G/H$ as in Theorem~\ref{thm:delannoy} is far smaller than $D_{\ell-1}$, since Theorem~\ref{thm:delannoy} bounds the number of complex solutions, not \emph{real, positive} solutions. Moreover, distinct real, positive solutions may correspond to isometric $\G$-invariant Einstein metrics; see Section~\ref{subsec:gaugegroup}. Bernstein's Theorem was previously used to estimate the number of homogeneous Einstein metrics on certain classes of homogeneous spaces by Graev~\cite{graev1,graev2,graev3}.

Remarkably, the homogeneous Einstein equations \eqref{eq:system} can be reinterpreted in the context of algebraic statistics~\cite{sullivant}. Namely, they are the critical equations of a maximum likelihood estimation problem on a scaled toric variety; see Theorem~\ref{thm:rewriteMLE}. We leverage previous work on the likelihood geometry of toric varieties \cite{ABB+} to prove our second main result (Theorem~\ref{thm:generic-finiteness}). In light of this, we believe that it would be fruitful to investigate further connections between algebraic statistics and geometric analysis on homogeneous spaces.

Fixing the monomials that appear in a given parametrized system of Laurent polynomial equations, 
there is a Zariski-dense subset of the space of coefficients for which the corresponding systems achieve the BKK bound. Coefficients that lie in this open set are called \emph{generic}, and the systems with those coefficients are said to be \emph{BKK generic}.
Since the parameters $L_{ijk}$ are symmetric in $i,j,k$, the coefficients in system \eqref{eq:system} are \emph{not} generic even for generic parameters ${\bf b}$, $\bf d$, and $L$.  
Thus, a priori, it is unclear if the bound in Theorem~\ref{thm:delannoy} is ever attained. Our second main result proves that it is attained, for generic parameters:

\begin{mainthm}\label{thm:generic-finiteness}
  If all parameters ${\bf b}$, ${\bf d}$, and $L$ are generic, then \eqref{eq:system} has exactly $D_{\ell - 1}$ solutions in $(\mathds C^*)^\ell$, counted with multiplicity.
  In particular, all solutions are isolated. 
\end{mainthm}

If the parameters are not generic, the number of isolated solutions drops. This can happen in two ways: either a solution in $(\CC^*)^\ell$ goes to zero or infinity, or a positive-dimensional component of solutions appears.
The subvariety of the parameter space where the BKK bound is \emph{not} achieved is called the \emph{BKK discriminant} and is described by \emph{Bernstein's Other Theorem} (Theorem~\ref{thm:other}). 
For parameters in the BKK discriminant, \eqref{eq:system} may admit infinitely many solutions even though the number of \emph{isolated} solutions drops. Next, we describe the BKK discriminant for \eqref{eq:system} in terms of the \emph{principal $A$-determinant} \eqref{eq:Adet} of the scalar curvature $\operatorname{scal}({\bf x})$ of the homogeneous metric $x_1\,Q|_{\mathfrak m_1}+\dots +x_\ell\,Q|_{\mathfrak m_\ell}$; see \eqref{eq:scalar}. This principal $A$-determinant $E_A(\operatorname{scal})$ is a polynomial in the parameters $\bf b$, $\bf d$, and $L$.

\begin{mainthm}\label{thm:bkkdiscriminant}
  The BKK discriminant of \eqref{eq:system} is contained in the zero set of
  \begin{equation}\label{eq:bkkdiscriminant}
    E_A(\operatorname{scal}) \cdot \prod_{S,T} \left( \sum_{i \in T} d_i + \sum_{j \in S}2d_j \right) 
  \end{equation}
  where $E_A(\operatorname{scal})$ is the principal $A$-determinant of scalar curvature, and the product is over nonempty $S, T \subseteq \{1,\ldots, \ell\}$ such that $S \cap T = \emptyset$. 
\end{mainthm}

Notably, Theorems~\ref{thm:generic-finiteness} and \ref{thm:bkkdiscriminant} yield sufficient (but not necessary) algebraic conditions on the parameters ${\bf b}$, ${\bf d}$, and~$L$ for \eqref{eq:system} to have only \emph{finitely many} solutions. In particular, the corresponding homogeneous spaces $\sf G/H$ have finitely many $\G$-invariant Einstein metrics. This gives a new perspective on a central open problem about homogeneous Einstein metrics:

\begin{fincon}[\cite{bwz}]
  If $M=\G/\mathsf H$ is a compact homogeneous space whose isotropy representation consists of pairwise inequivalent irreducible summands,
  e.g., when $\operatorname{rank}\G=\operatorname{rank}\mathsf H$,
  then the Einstein equations \eqref{eq:system} have only finitely many real solutions.
\end{fincon}

Our sufficient algebraic conditions for finiteness can be stated as follows: 

\begin{maincor}\label{cor:fin}
  Let $\G/{\sf H}$ be a compact homogeneous space whose isotropy representation consists of $\ell$ pairwise inequivalent $\mathsf H$-irreducible summands, with associated parameters $\bf b$, $\bf d$, and $L$.
  If the principal $A$-determinant $E_A(\operatorname{scal})$ does not vanish on $\bf b$, $\bf d$, and $L$, then there are at most $D_{\ell - 1}$ many $\G$-invariant Einstein metrics on $\sf G/H$. In particular, the Finiteness Conjecture holds on~$\G/{\sf H}$.
\end{maincor}

Principal $A$-determinants are generally difficult to compute.
For $\ell = 2,3$, the principal $A$-determinants of $\operatorname{scal}$ are found in Proposition~\ref{prop:l=2bkk} and \eqref{eq:l=3-gen-disc}, respectively. In the special case $\ell = 3$ and $L_{iik} = 0$ for $i\neq k \in \{1,2,3\}$, the Laurent polynomial $\operatorname{scal}$ has a different principal $A$-determinant, that is computed in Proposition~\ref{prop:l=3}; in this case, the BKK bound drops from $D_2=13$ to 4. As an application, we show that there are at most 4 distinct homogeneous Einstein metrics on the generalized Wallach spaces (Theorem~\ref{thm:CN-finiteness}), providing an alternative proof of \cite[Thm.~1]{gen-wallach-einstein}. Most of these systems achieve their BKK bound~of~4.

On the other hand, we also find examples of homogeneous spaces for which the BKK bound for \eqref{eq:system} is not achieved; see Sections~\ref{sec:applications} and \ref{sec:numerics}. 
This shows that establishing BKK genericity is not a viable option to prove the Finiteness Conjecture in full generality. 
We compute numerically the solutions to \eqref{eq:system} in some of these examples. 

\begin{mainthm}\label{thm:examples}
  The number of solutions to \eqref{eq:system} for low-dimensional full flag manifolds $\sf G/H$, where $\G$ is a compact simple Lie group of type ${\rm A}_n$, ${\rm B}_n$, ${\rm C}_n$, or ${\rm D}_n$ and $\sf H\subset\G$ is a maximal torus, are found in Table~\ref{tab:numerics}. 
  In particular, up to isometries, there are at least
  \begin{enumerate}
  \item[\rm ($\rm A_4$)] $12$ homogeneous Einstein metrics on ${\sf SU}(5)/\sf T^4$ (see \Cref{tab:A4}),
  \item[\rm ($\rm A_5$)] $35$ homogeneous Einstein metrics on ${\sf SU}(6)/\sf T^5$ (see code accompanying this paper),
  \item[\rm ($\rm B_3$)] $5$ homogeneous Einstein metrics on ${\sf SO}(7)/\sf T^3$ (see \Cref{tab:B3}),
  \item[\rm ($\rm C_3$)] $4$ homogeneous Einstein metrics on ${\sf Sp}(3)/\sf T^3$ (see \Cref{tab:C3}),
  \item[\rm ($\rm D_4$)]  $5$ homogeneous Einstein metrics on ${\sf SO}(8)/\sf T^4$ (see \Cref{tab:D4}).
  \end{enumerate}
  In all of the above cases, the BKK bound for the system \eqref{eq:system} is not achieved.
\end{mainthm}
Our numerical methods give rigorous lower bounds on the number of solutions to a system; see Section~\ref{subsec:nag}.
We conjecture that the counts in Theorem~\ref{thm:examples} are, in fact, equal to the true number of homogeneous Einstein metrics on these spaces, up to isometries.
With the exception of the space ${\sf SO}(8)/\sf T^4$, the solution counts above were previously computed  with different methods; see \cite{GM16,GW,WLZ18}.

This paper is organized as follows.
Background on homogeneous Einstein metrics is discussed in Section~\ref{sec:diffgeom}.
In Section~\ref{sec:bridge}, we explicitly describe the Newton polytopes of \eqref{eq:system}
and compute the BKK bound in Theorem~\ref{thm:delannoy}. In Section~\ref{sec:algstat}, we explain how to interpret \eqref{eq:system} in the context of algebraic statistics and prove Theorems~\ref{thm:generic-finiteness} and \ref{thm:bkkdiscriminant}. 
In Section~\ref{sec:applications}, we study \eqref{eq:system} on generalized Wallach spaces.
Finally, our computations on full flag manifolds are described in Section \ref{sec:numerics}.
The code used for these computations is made available at:

\smallskip
\centerline{\url{https://github.com/hannahfriedman/counting_homogeneous_einstein_metrics}.}

\subsection*{Notation} For the readers' convenience, we collect here all basic notation used in the paper. 
We write $\RR$, $\CC$, $\mathds H$, and $\mathds{C}\mathrm{a}$ for the real division algebras of reals, complex numbers, quaternions, and octonions, respectively. 
We write $[n]=\{1,\dots,n\}$ for natural numbers~$n\in \mathds N$. Vectors ${\bf v}=(v_1,\dots,v_n)^T$ are written in boldface, and $\operatorname{diag}(\bf v)$ denotes the $n\times n$ diagonal matrix with entries $v_i$. We write ${\bf e}_i \in \RR^n$ for the $i$th column of the $n\times n$ identity matrix, and set ${\bf 1}=(1,\dots,1)^T$ and ${\bf 0}=(0,\dots,0)^T$.
For ${\bf x}=(x_1,\dots,x_\ell)^T$ and ${\bf a}=(a_1,\dots,a_\ell)^T$, we write ${\bf x}^{{\bf a}}= x_1^{a_{1}} \cdots x_\ell^{a_{\ell}}$.
Given an $\ell\times r$ matrix $A=({\bf a}_1 \cdots {\bf a}_\ell)\in \mathds Z^{\ell\times r}$ and ${\bf x}\in \CC^\ell$, we set ${\bf x}^A = ({\bf x}^{{\bf a}_1} , \ldots , {\bf x}^{{\bf a}_r})^T\in \CC^r$. 
We write
$(\CC^*)^\ell=(\CC\setminus\{0\})^\ell$ for the $\ell$-dimensional algebraic torus and $\mathds P_{\CC}^{n-1}$ for the complex projective $(n-1)$-space; projective coordinates are denoted $(z_1:\dots:z_{n})$. Given Lie groups $\mathsf H\subset\G$ with Lie algebras $\mathfrak h\subset\mathfrak g$, we denote by 
$\mathrm{Ad}_{\mathsf H}$ the adjoint representation of $\mathsf H$ on $\mathfrak g$, given by $\mathrm{Ad}_h X=\frac{\dd}{\dd t}h(\exp tX)h^{-1}|_{t=0}$, for all $h\in\mathsf H$, $X\in \mathfrak g$.

\subsection*{Acknowledgments}
We are grateful to Bernd Sturmfels for introducing the authors and providing feedback at various stages. 
We thank Andr\'{e}s R.\ Vindas Mel\'{e}ndez for bringing the reference \cite{postnikov} to our attention, and Wolfgang Ziller for several conversations on homogeneous Einstein metrics and the Finiteness Conjecture.
The first-named author is supported by the National Science Foundation CAREER grant DMS-2142575. 

\section{Homogeneous Einstein Metrics}\label{sec:diffgeom}

In this section, we discuss basic facts about compact homogeneous spaces, including the equations satisfied by homogeneous Einstein metrics; for further details; see \cite{besse,mybook}. 

\subsection{Setup}
Let $(M,\mathrm g)$ be a compact homogeneous space, that is, a compact Riemannian manifold endowed with a transitive isometric action by a (compact) Lie group $\G$. Let $\mathsf H\subset \mathsf G$ be the isotropy subgroup of a point $p_0\in M$, so that $M=\mathsf G/\mathsf H$, and let $\mathfrak h\subset\mathfrak g$ be the Lie algebras of $\mathsf H\subset \mathsf G$. 
Fix a bi-invariant metric $Q$ on $\G$, see \cite[Prop.~2.24]{mybook}, and a $Q$-orthogonal complement $\mathfrak m$ to $\mathfrak h\subset \mathfrak g$. Then $\mathfrak m$ can be identified with the tangent space $T_{p_0}M$ by associating to each $X\in \mathfrak m$ the action vector field $X^*_{p_0}= \frac{\dd}{\dd t} \exp(tX)\cdot p_0 \big|_{t=0}\in T_{p_0}M$. Using this identification, one shows that the evaluation map $\mathrm g\mapsto \mathrm g_{p_0}$ determines a bijection between the set of $\mathsf G$-invariant Riemannian metrics $\mathrm g$ on $M$ and the set of $\operatorname{Ad}_{\mathsf H}$-invariant inner products on $\mathfrak m\cong T_{p_0} M$; see \cite[Thm.~6.13]{mybook}. 

\begin{remark}
  Some manifolds $M$ admit several (even infinitely many) presentations as a homogeneous space $M=\G_1/\mathsf H_1=\G_2/\mathsf H_2=\dots$, corresponding to different transitive actions on $M$, even of the same group $\G$; see, e.g.,~\cite[Ex.~6.9]{bwz}. When we discuss $\G$-invariant metrics on $M$, we implicitly fix the transitive $\G$-action on $M=\G/\mathsf H$.
\end{remark}

We now specialize to a subclass of $\G$-invariant metrics on $M$. Following the notation of Wang and Ziller~\cite{wang-ziller-inv}, let
\begin{equation}\label{eq:splitting-m}
  \mathfrak m=\mathfrak m_1\oplus \dots \oplus \mathfrak m_r\oplus \dots \oplus  \mathfrak m_\ell
\end{equation}
be a $Q$-orthogonal decomposition into $\operatorname{Ad}_{\mathsf H}$-invariant subspaces, so that $\operatorname{Ad}_{\mathsf H}$ acts irreducibly on $\mathfrak m_i$ for $1\leq i\leq \ell$, and trivially on $\mathfrak m_{i}$ for $r<i\leq \ell$. A $\G$-invariant Riemannian metric $\mathrm g$ on $M$ is \emph{diagonal} for the decomposition \eqref{eq:splitting-m} if it is induced by an inner product 
of the form
\begin{equation}\label{eq:diagonal-metric}
\phantom{, \qquad x_i>0.}
  \mathrm g_{p_0}=x_1 \, Q|_{\mathfrak m_1 } + \dots + x_\ell \, Q|_{\mathfrak m_\ell }, \qquad x_i>0,
  \end{equation}
and we write ${\bf x}=(x_1,\dots,x_\ell)$.
Here, as customary, we identify any bilinear form $\langle\cdot,\cdot\rangle$ on the vector space $V$ with the linear map $V\to V^*$ given by $v\mapsto \langle v,\cdot\rangle$.
In other words, diagonal metrics for \eqref{eq:splitting-m} are those for which the $\mathfrak m_i$ are pairwise orthogonal. For instance, if the $\mathfrak m_i$ are pairwise inequivalent nontrivial $\operatorname{Ad}_{\mathsf H}$-representations, that is, $r=\ell$ and $\mathfrak m_i\not\cong \mathfrak m_j$ for all $i\neq j$, then all $\mathsf G$-invariant metrics on $M$ are diagonal. This is the default situation we consider in this paper. In general, every $\mathsf G$-invariant metric on $M$ is diagonal \emph{for some} decomposition of the form \eqref{eq:splitting-m}; see Wang and Ziller~\cite[p.~180]{wang-ziller-inv}. However, for a fixed decomposition \eqref{eq:splitting-m}, if at least two of the $\mathfrak m_i$ are equivalent or trivial, then there exist $\G$-invariant metrics on $\G/\mathsf H$ that are \emph{not} diagonal with respect to that decomposition. 

\subsection{Homogeneous Einstein equations}
We now introduce constants associated to a compact homogeneous space $\mathsf G/\mathsf H$, in order to write the homogeneous Einstein equations.
Let $b_i\in \RR$ be the constants so that the Cartan--Killing form $B(X,Y)=\operatorname{tr} (\operatorname{ad}_X\circ\operatorname{ad}_Y)$ satisfies 
\[B|_{\mathfrak m_i }=-b_i\,Q|_{\mathfrak m_i },\]
and set $d_i\coloneqq \dim \mathfrak m_i$; we collect these as vectors ${\bf b}=(b_1,\dots,b_\ell)^T$ and ${\bf d}=(d_1,\dots,d_\ell)^T$.
Recall that $b_i\geq0$ and $b_i=0$ if and only if $\mathfrak m_i\subset \operatorname{Z}(\mathfrak g)$; see \cite[Thm.~2.35, Cor.~2.46]{mybook}. Next, let $\{{\bf v}_\alpha\}$ be a $Q$-orthonormal basis of $\mathfrak m$ adapted to \eqref{eq:splitting-m}, that is, a basis satisfying the condition that for all
${\bf v}_\alpha$ and ${\bf v}_\beta$, 
there exist $i,j\in [\ell]$ such that ${\bf v}_\alpha\in\mathfrak m_i$ and ${\bf v}_\beta\in\mathfrak m_j$; furthermore, if $i<j$, then $\alpha<\beta$.
Define the \emph{structure constants}
\begin{equation*}
  L_{ijk} \coloneqq \sum_{\substack{{\bf v}_\alpha \in \mathfrak m_i \\ {\bf v}_\beta \in \mathfrak m_j \\ {\bf v}_\gamma \in \mathfrak m_k }} Q([{\bf v}_\alpha,{\bf v}_\beta],{\bf v}_\gamma)^2.
\end{equation*}
Note that $L_{ijk}$ does not depend on the choice of $Q$-orthonormal basis, but only on the decomposition \eqref{eq:splitting-m}. The constants $L_{ijk}$ are nonnegative and $L_{ijk}=0$ if and only if $[\mathfrak m_i,\mathfrak m_j]$ is $Q$-orthogonal to $\mathfrak m_k$. Moreover, $L_{ijk}$ is symmetric in its $3$ indices.

The Ricci tensor of the diagonal metric $\mathrm g$ on $M$ satisfying \eqref{eq:diagonal-metric} is uniquely determined by its value at $p_0$, which, just like $\mathrm g_{p_0}$, is an $\operatorname{Ad}_{\mathsf H}$-invariant symmetric bilinear form on $\mathfrak m$. If the $\mathfrak m_i$ are pairwise inequivalent, then, by Schur's Lemma, $(\operatorname{Ric}_{\mathrm g})_{p_0}$ is also diagonal with respect to \eqref{eq:splitting-m}, so it can be written as
\begin{equation}\label{eq:diagonal-Ric}
\begin{aligned}
  (\operatorname{Ric}_{\mathrm g})_{p_0}&=r_1^\ell({\bf x})x_1 \, Q|_{\mathfrak m_1 }  +  \dots  + r^\ell_\ell({\bf x})x_\ell \, Q|_{\mathfrak m_\ell },\\
  &= r^\ell_1({\bf x})\, {\mathrm g}_{p_0}|_{\mathfrak m_1 } +  \dots + r^\ell_\ell({\bf x}) \, {\mathrm g}_{p_0}|_{\mathfrak m_\ell },
\end{aligned}
\end{equation}
for some $r_i^\ell({\bf x})$. 
Direct computation, see e.g.~\cite[Cor.~7.38]{besse} or \cite[Lem.~1.1]{pas}, gives
\begin{equation}\label{eq:Ricci-entries}
\phantom{\qquad \quad  1 \leq i \leq \ell}
  r^\ell_i({\bf x})  = \frac{b_i}{2x_i} - \frac{1}{4 d_i} \sum_{j,k = 1}^\ell L_{ijk} \frac{2x_k^2 - x_i^2}{x_ix_jx_k},
  \qquad \quad  1 \leq i \leq \ell.
\end{equation}
In this situation, the diagonal metric $\mathrm g$ is Einstein if and only if $r_i^\ell({\bf x})=r_j^\ell({\bf x})$ for all $i,j\in  [\ell]$, and its Einstein constant is the common value $\lambda=r_i^\ell({\bf x})$.

The Ricci tensor is invariant under homotheties of the metric, $\operatorname{Ric}_{\alpha \mathrm g}=\operatorname{Ric}_{\mathrm g}$ for all $\alpha>0$; correspondingly, the $r_i^\ell$ are homogeneous of degree $-1$, that is, $r_i^\ell(\alpha {\bf x}) = \tfrac1\alpha r_i^\ell({\bf x})$ for all $\alpha>0$ and $i\in[\ell]$. Thus, it is customary to normalize the Einstein constant as $\lambda=1$, which leads to the system \eqref{eq:system} of equations $r_i^\ell({\bf x})=1$ for all $i\in [\ell]$.

\subsection{Isometries and gauge group}\label{subsec:gaugegroup}
Two homogeneous metrics $\mathrm g$ as in \eqref{eq:diagonal-metric} with different values of ${\bf x}\in \RR^\ell_+$ may be \emph{isometric}, that is, obtained from one another via pullback by a diffeomorphism of $M$. Isometric Riemannian metrics are indistinguishable geometrically, so it is desirable to count solutions ${\bf x}\in \RR^\ell_+$ to \eqref{eq:system} only \emph{up to isometries}. Detecting such isometries is, in general, a hard problem. A sufficient condition for two metrics to be  nonisometric is that some geometric invariant, e.g., the volume 
\begin{equation}\label{eq:vol}
 \operatorname{Vol}(M,{\mathrm g}) = \operatorname{Vol}(M,{Q|_{\mathfrak m}}) \prod_i x_i^{d_i},
\end{equation}
assumes different values on them. Other geometric invariants, such as the diameter and Laplace spectrum, could be used as well, but these are often quite difficult to compute, even on compact homogeneous spaces.

Some isometries between $\G$-invariant metrics on $M=\G/\mathsf H$ are easy to describe. Each element $n\in \mathsf N(\mathsf H)$ in the normalizer of $\mathsf H$ in $\G$ determines a $\G$-equivariant diffeomorphism $\phi_n\colon \G/\mathsf H\to \G/\mathsf H$, given by $\phi_n(g\mathsf H) = gn\mathsf H$. This induces a free action of the \emph{gauge group} $\mathsf N(\mathsf H)/\mathsf H$ on $M$, and allows us to identify $\mathsf N(\mathsf H)/\mathsf H$ with the group of $\G$-equivariant diffeomorphisms of $M$. 
This group then acts (via pullback) on the space of $\G$-invariant metrics on $M$:
if $\mathrm g$ is a $\G$-invariant metric on $M$, then so is $\phi_n^*\,\mathrm g$, and these are, by definition, isometric. In particular, for a diagonal metric $\mathrm g$ determined by ${\bf x}\in\RR^\ell_+$ as in \eqref{eq:diagonal-metric}, we have 
\begin{equation*}
\phantom{\text{ for all }X,Y\in\mathfrak m}
  (\phi_n^*\,\mathrm g)_{p_0}(X^*_{p_0},Y^*_{p_0})=Q(\mathrm{Ad}_n \, \mathrm{diag}({\bf x})\, \mathrm{Ad}_n^{-1} X,Y), \;\; \text{ for all }X,Y\in\mathfrak m.
\end{equation*}
Note that, in general, $\phi_n^*\,\mathrm g$ need not be diagonal. But if the $\mathfrak m_i$ are pairwise inequivalent, then $\phi_n^*\,\mathrm g$ is diagonal and, as in \eqref{eq:diagonal-metric}, it corresponds to an $\ell$-tuple $\sigma\cdot{\bf x}=(x_{\sigma(i)})\in\RR^\ell_+$ obtained from ${\bf x}\in\RR^\ell_+$ via a permutation $\sigma$ on $[\ell]$. In this case, the gauge group $\mathsf N(\mathsf H)/\mathsf H$ is finite, which explains why the Finiteness Conjecture of \cite{bwz} is only stated for pairwise inequivalent~$\mathfrak m_i$.
There are examples of $\G/\mathsf H$ for which some of the $\mathfrak m_i$ are equivalent and there are positive-dimensional components of $\G$-invariant Einstein metrics with $\lambda=1$; see \cite[Ex.~6.10]{bwz}. However, in these examples, such  components are orbits of the gauge group $\mathsf N(\mathsf H)/\mathsf H$, which has positive dimension, and there are still only \emph{finitely many} $\G$-invariant Einstein metrics with $\lambda=1$, \emph{up to isometries}.

\subsection{Examples}
Let us discuss the homogeneous Einstein equations \eqref{eq:system} on some examples in which the homogeneous space $M=\G/\mathsf H$ is a sphere; these were studied by Ziller~\cite{ziller-mathann}.

\begin{example}[Berger spheres, $\CC$]\label{ex:bergerC}
For $n\geq1$, consider the unit sphere $S^{2n+1}\subset \CC^{n+1}$ endowed with the transitive action of $\mathsf G=\mathsf{SU}(n+1)$. The isotropy of $p_0=(0,\dots,0,1)$ consists of the block diagonal matrices $\mathsf H = \{\operatorname{diag}(A,1)\in\mathsf G : A\in \mathsf{SU}(n) \}$. We endow $\mathfrak g=\mathfrak{su}(n+1)$ with the standard bi-invariant metric $Q(X,Y)=-\frac12 \operatorname{Re}\operatorname{tr}XY$, and recall (see, e.g., \cite[Ex.~6.16]{mybook}) that the $Q$-orthogonal complement $\mathfrak m$ to $\mathfrak h\subset\mathfrak g$ splits as $\mathfrak m = \mathfrak m_1\oplus \mathfrak m_2$, where the $\operatorname{Ad}_{\mathsf H}$-representation $\mathfrak m_1\cong\CC^n$ is the defining representation and $\mathfrak m_2\cong\RR$ is trivial.

In this case, the various parameters discussed above can be computed to be 
  \begin{equation*}
  \ell=2, \quad {\bf d }=(2n, 1),\quad  {\bf b}=(4n+4)\,{\bf 1},  \quad L_{112}=4n+4,\;\; L_{111}=L_{122}=L_{222}=0,  
  \end{equation*}
  so the system \eqref{eq:system} reduces to
  \begin{align*}
    \frac{(n+1)}{n}\frac{x_2}{x_1}+x_1 &=2 n+2,\\
    (n+1)\frac{x_2}{x_1^2}&=1,
  \end{align*}
  which admits a unique solution ${\bf x}=2n\,\big(1,\frac{2n}{n+1}\big)$. 
  This is the round metric of radius $\sqrt{2n}$, which is known to be the only $\mathsf G$-invariant Einstein metric on $S^{2n+1}$; see Ziller~\cite{ziller-mathann}.
\end{example}

\begin{example}[Berger spheres, $\HH$]\label{ex:bergerH}
  For $n\geq1$, consider the unit sphere $S^{4n+3}\subset \HH^{n+1}$ endowed with the transitive action of $\mathsf G=\mathsf{Sp}(n+1)$. The isotropy of $p_0=(0,\dots,0,1)$ consists of the block diagonal matrices $\mathsf H = \{\operatorname{diag}(A,1)\in\mathsf G : A\in \mathsf{Sp}(n) \}$. We endow $\mathfrak g=\mathfrak{sp}(n+1)$ with the standard bi-invariant metric $Q(X,Y)=-\frac12 \operatorname{Re}\operatorname{tr}XY$, and recall (see, e.g., \cite[Ex.~6.16]{mybook}) that the $Q$-orthogonal complement $\mathfrak m$ to $\mathfrak h\subset\mathfrak g$ splits as $\mathfrak m = \mathfrak m_1\oplus \mathfrak m_2\oplus \mathfrak m_3\oplus \mathfrak m_4$, where the $\operatorname{Ad}_{\mathsf H}$-representation $\mathfrak m_1\cong\HH^n$ is the defining representation and $\mathfrak m_i\cong\RR$, $i=2,3,4$ are trivial.
  In this case, the parameters are:
  \[\ell=4, \quad {\bf d}=(4n,1,1,1),\quad {\bf b}=(8n+16) \,{\bf 1}, \quad L_{112}=L_{113}=L_{114}=8n,\; L_{234}=8, \]
  and all other $L_{ijk}$, $i\leq j\leq k$, vanish, so the polynomial system \eqref{eq:system} becomes
  \begin{align*}
  \frac{4 n+8}{x_1}- \frac{ 2 x_2^2-x_1^2 }{2x_1^2 x_2}-\frac{ 2 x_3^2-x_1^2 }{2x_1^2 x_3}-\frac{ 2 x_4^2-x_1^2 }{2x_1^2 x_4}-\frac{1}{2x_2}-\frac{1}{2x_3}-\frac{1}{2x_4} &=1,\\
  \frac{4 n+8}{x_2}- \frac{2 n \left(2 x_1^2-x_2^2\right)}{x_1^2 x_2} - \frac{2 \left(2 x_3^2-x_2^2\right)}{x_2 x_3 x_4} - \frac{2 \left(2 x_4^2-x_2^2\right)}{x_2 x_3 x_4}&=1,\\
  \frac{4 n+8}{x_3}-\frac{2 n \left(2 x_1^2-x_3^2\right)}{x_1^2 x_3} - \frac{2\left(2 x_2^2-x_3^2\right)}{x_2 x_3 x_4} - \frac{2 \left(2 x_4^2-x_3^2\right)}{x_2 x_3 x_4}&=1,\\
  \frac{4n+8}{x_4}- \frac{2 n \left(2 x_1^2-x_4^2\right)}{x_1^2 x_4} - \frac{2 \left(2 x_2^2-x_4^2\right)}{x_2 x_3 x_4} - \frac{2 \left(2 x_3^2-x_4^2\right)}{x_2 x_3 x_4}&=1.
  \end{align*}
This system admits $8$ solutions in $(\mathds{C}^*)^4$ for generic $n$. Of these $8$ solutions, only two are positive: ${\bf x}=(4n+2)\,(1,2,2,2)$, which is  the round metric of radius $\sqrt{4n+2}$, and ${\bf x}=\frac{8n^2+28n+18}{2n+3}\big(1,\frac{2}{2n+3},\frac{2}{2n+3},\frac{2}{2n+3}\big)$, which is the Jensen metric~\cite{jensen}. It was shown by Ziller~\cite{ziller-mathann} that these are the only $\mathsf G$-invariant Einstein metrics on $S^{4n+3}$ for any $n \geq1$.

Since $\mathfrak m$ contains 3 copies of the trivial representation, $S^{4n+3}$ also admits nondiagonal $\G$-invariant metrics. However, using the subgroup $\mathsf{Sp}(1)=\{\operatorname{diag}(\mathrm{Id},q)\in\mathsf G : q\in \mathsf{Sp}(1) \}$ of the gauge group ${\sf N(H)/H}$, one shows that every nondiagonal metric is isometric to some diagonal metric. Thus, in this case, no generality is lost in considering only diagonal metrics.
\end{example}

For further examples, see \Cref{ex:somn} and Sections~\ref{sec:applications} and \ref{sec:numerics}.

\subsection{\texorpdfstring{Case $\ell=2$}{Two summands}}\label{subsec:ell=2}
Compact homogeneous spaces $\mathsf G/\mathsf H$ whose isotropy representation consists of $\ell=2$ irreducible summands are classified \cite{dk1,dk2}. On such a space, by \eqref{eq:Ricci-entries}, the metric \eqref{eq:diagonal-metric} is Einstein with Einstein constant $\lambda$ if and only if ${\bf x}\in \RR^2_+$ satisfies the system
  \begin{equation}\label{eq:l=2}
  \begin{aligned}
    r_1^2({\bf x}) = \frac{b_1}{2x_1} - \frac{1}{4d_1} \left(\frac{L_{111}}{x_1} + \frac{2L_{112}x_{2}}{x_1^2} + \frac{2L_{122}}{x_1} - \frac{L_{122}x_1}{x_2^2}\right) &= \lambda, \\ 
    r_2^2({\bf x}) = \frac{b_2}{2x_2} - \frac{1}{4d_2} \left(\frac{L_{222}}{x_2} + \frac{2L_{122}x_{1}}{x_2^2} + \frac{2L_{112}}{x_2} - \frac{L_{112}x_2}{x_1^2}\right) &= \lambda.
  \end{aligned}
  \end{equation}
We perform symbolic elimination on $\lambda$ by setting the two left-hand sides of \eqref{eq:l=2} equal to one another. After clearing denominators, the above can be
rewritten as the cubic
\begin{multline}\label{eq:cubic}
(2 d_1 +d_2) L_{122} x_1^3 + d_1(2 L_{112}+ L_{222}-2 b_2 d_2) x_1^2x_2 \\ -d_2 (L_{111}+2 L_{122}-2 b_1 d_1) x_1 x_2^2
    -(d_1 + 2 d_2) L_{112} x_2^3 =0.
\end{multline}
If the coefficients $(2d_1 + d_2)L_{122}$ and $(d_1 + 2d_2)L_{112}$ are nonzero, then \eqref{eq:cubic} has exactly $3$ solutions in $\mathds{P}_\CC^1$, counted with multiplicity.
If the polynomial $x_1^2x_2^2 r_1^2({\bf x}) = x_1^2x_2^2 r_2^2({\bf x})$ does not vanish on a solution to \eqref{eq:cubic}, then there exists a representative ${\bf x}^* \in (\mathds{C}^*)^2$ of that solution such that $r_1^2({\bf x}^*) = r_2^2({\bf x}^*) = 1$ and hence \eqref{eq:l=2} is satisfied with ${\bf x} = {\bf x}^*$ and $\lambda =1 $. 
Therefore if $(2d_1 + d_2)L_{122}$ and $(d_1 + 2d_2)L_{112}$ are nonzero and $x_1^2x_2^2r_1^2({\bf x}) = x_1^2x_2^2 r_2^2({\bf x})$ does not vanish on any of the three solutions to \eqref{eq:cubic}, then \eqref{eq:l=2} with $\lambda = 1$ has $D_1 = 3$ solutions in $(\CC^*)^2$, counted with multiplicity, achieving the BKK bound in Theorem~\ref{thm:delannoy}.

The discussion above proves a more explicit version of Theorem~\ref{thm:bkkdiscriminant} in the case $\ell = 2$:

\begin{proposition}\label{prop:l=2bkk}
  The system~\eqref{eq:l=2} with $\lambda = 1$ has exactly 3 solutions in $(\CC^*)^2$, counted with multiplicity, if and only if
  \[(2d_1 + d_2)\, (d_1 + 2d_2) \, R(r_1^2, r_2^2) \neq 0,\] where $R(r_1^2, r_2^2)$ is the resultant of the polynomials $x_1^2x_2^2\,r^2_1({\bf x})$ and $x_1^2x_2^2\,r^2_2({\bf x})$, i.e., the determinant of the  Sylvester matrix
\begin{equation*}
  \begin{bmatrix}
    L_{122} & L_{222}' & L_{111}' & L_{112} & & \\
    & L_{122} & L_{222}' & L_{111}' & L_{112} & \\
    & & L_{122} & L_{222}' & L_{111}' & L_{112} \\
    3L_{122} & 2L_{222}' & L_{111}' \\
    & 3L_{122} & 2L_{222}' & L_{111}' \\
    & & 3L_{122} & 2L_{222}' & L_{111}' \\    
  \end{bmatrix},
\end{equation*}
where $L'_{111}=L_{111}+2 L_{122} - 2 b_1 d_1$ and $L'_{222}=L_{222}+2 L_{112} - 2 b_2 d_2$; see \eqref{eq:Lprime}.
\end{proposition}

The condition that \eqref{eq:l=2} with $\lambda = 1$ has finitely many solutions in $(\mathds{C}^*)^2$ is weaker, namely, it is equivalent to the condition that some coefficient of \eqref{eq:cubic} is nonzero.
Note that there are nonnegative choices of parameters ${\bf b}, {\bf d}, L$ such that all coefficients of \eqref{eq:cubic} become zero:

\begin{example}
  Set ${\bf b} = (14,\, 12)$, ${\bf d} = (10,\,15)$, $L_{111} = 280$, $L_{112} = 0$, $L_{122} = 0$, $L_{222} = 360$. All coefficients of \eqref{eq:cubic} vanish, so \eqref{eq:cubic} is satisfied for any $\bf x$, and the equations $r_1^2({\bf x}) = r_2^2({\bf x}) = 0$ vanish identically. Thus, \eqref{eq:cubic} has infinitely many solutions, which are solutions to \eqref{eq:l=2} with $\lambda=0$. These values of the parameters do not correspond to any compact homogeneous space: homogeneous Ricci flat metrics are flat~\cite[Thm.~7.61]{besse}, so they only arise if $L\equiv0$.
\end{example}

We now show that \eqref{eq:l=2} can only have infinitely many solutions if $\lambda=0$, as in the above example, provided $d_1,d_2>0$. In particular, this implies the Finiteness Conjecture for $\ell = 2$.

\begin{proposition}\label{prop:l=2finite}
  The system \eqref{eq:l=2} with $\lambda = 1$ has finitely many solutions in $(\CC^*)^2$ provided that 
  $d_1$, $d_2$, $2d_1 + d_2$, and $d_1 + 2d_2$ are nonzero; in particular, this holds if $d_1, d_2 > 0$. 
\end{proposition}

\begin{proof}
  If \eqref{eq:l=2} with $\lambda = 1$ has infinitely solutions in $(\mathds{C}^*)^2$, then \eqref{eq:cubic} has infinitely many solutions in $\mathds{P}_{\CC}^1$ and hence \eqref{eq:cubic} vanishes identically, i.e., all its coefficients are zero. It then follows from our assumptions that $L_{112}=  L_{122}=0$, $L_{111}=2b_1d_1$, and $L_{222}=2b_2d_2$.
 So, plugging any ${\bf x} \in (\mathds{C}^*)^2$ into \eqref{eq:l=2}, we obtain the contradiction
    \[ \lambda = r_1^2({\bf x})=\frac{b_1}{2x_1} - \frac{1}{4d_1} \left( \frac{L_{111}}{x_1} \right ) =
    \frac{2b_1d_1 - L_{111}}{4d_1x_1} = 0.\qedhere\]
\end{proof}

\begin{remark}
  The polynomial \eqref{eq:cubic} is used in \cite[Thm.~3.1]{wang-ziller-inv} to prove that certain compact homogeneous spaces $\mathsf G/\mathsf H$ with $\ell=2$ admit no $\G$-invariant Einstein metrics; see also \cite{dk1,dk2}. This is done assuming that there is an intermediate Lie group $\mathsf H\subsetneq\mathsf K\subsetneq \mathsf G$, so either $\mathfrak h\oplus\mathfrak m_1$ or $\mathfrak h\oplus\mathfrak m_2$ is a Lie subalgebra of $\mathfrak g$, hence either $L_{112} = 0$ or $L_{122}=0$. In this case, \eqref{eq:cubic} reduces to a quadric, so there are no real solutions if its discriminant is negative. If, instead, the subgroup $\mathsf H\subset\G$ is maximal, then \eqref{eq:cubic} is actually a cubic and there exist $\G$-invariant Einstein metrics on $\G/\mathsf H$; see, e.g.,~\cite[Thm.~2.2]{wang-ziller-inv}.
\end{remark}

\section{The BKK bound for the homogeneous Einstein equations}\label{sec:bridge}

\subsection{BKK bound}\label{subsec:bkk_bound}
Let $\mathcal F = \{f_1, \ldots, f_\ell\}$ be a system of Laurent polynomials in $\ell$ variables.
The \emph{support} of $f_i ({\bf x})= \sum_{{\bf a} \in \ZZ^\ell} c_{i,\bf a} {\bf x}^{\bf a}$ is the finite set ${\rm supp}(f) = \{{\bf a} \in \ZZ^\ell \colon c_{i,\bf a} \neq 0\}$. 
The \emph{Newton polytope} of $f_i$ is the convex hull of the support, i.e., $P_i = {\rm Newt}(f_i) = \conv({\rm supp}(f_i))$ for each $i \in [\ell]$.
Given $\lambda_1,\dots,\lambda_\ell>0$, the $\ell$-dimensional volume of the scaled Minkowski sum $\lambda_1 P_1 + \cdots + \lambda_\ell P_\ell\subset\RR^\ell$ is a polynomial function of the $\lambda_i$'s; the coefficient of $\lambda_1 \cdots \lambda_\ell$ in this polynomial is called the \emph{mixed volume} ${\rm MV}(P_1, \ldots, P_\ell)$; see, e.g.,~\cite[\S7.4]{CLO} for details. We also refer to ${\rm MV}(P_1, \ldots, P_\ell)$ as the mixed volume of the system $\mathcal F$. The following result establishes the so-called \emph{Bernstein-Khovanskii-Kushnirenko (BKK) bound}.

\begin{theorem}[Bernstein {\cite[Thm.~A]{bernstein}}]\label{thm:bkk}
  The system $\mathcal F$ has at most ${\rm MV}(P_1, \ldots, P_\ell)$ isolated solutions in $(\CC^*)^\ell$, counted with multiplicity.
  If the coefficients of $\mathcal F$ are generic, then $\mathcal F$ has exactly ${\rm MV}(P_1, \ldots, P_\ell)$ many solutions in $(\CC^*)^\ell$, all of which are isolated. 
\end{theorem}

The proof of Theorem~\ref{thm:delannoy} is divided in two parts. First, in Section~\ref{sub:polytopes}, we prove that the mixed volume of the system \eqref{eq:system} is equal to the normalized volume of a single polytope:

\begin{theorem}\label{thm:mv-vol}
  The mixed volume of the system \eqref{eq:system} is equal to the normalized volume $\ell! \operatorname{Vol}(\widetilde P^\ell)$ of the permutohedron $\widetilde P^\ell = \conv({\bf 0}, {\bf e}_i - 2{\bf e}_j : i,j \in [\ell])$.
\end{theorem}

Then, in Section~\ref{sub:volume}, we compute this normalized volume explicitly:

\begin{theorem}\label{thm:centraldelannoy}
  The normalized volume $\ell! \operatorname{Vol}(\widetilde P^\ell)$ is the central Delannoy number $D_{\ell -1}$. 
\end{theorem}

Combining the above results,
 we obtain Theorem~\ref{thm:delannoy}:

\begin{proof}[Proof of Theorem~\ref{thm:delannoy}]
  By Theorems~\ref{thm:bkk} and \ref{thm:mv-vol}, the number of isolated solutions to \eqref{eq:system} in $(\CC^*)^\ell$, counted with multiplicity, is bounded above by the normalized volume $\ell! \operatorname{Vol}(\widetilde P^\ell)$. By Theorem~\ref{thm:centraldelannoy}, this normalized volume is equal to $D_{\ell-1}$. Among these solutions in $(\CC^*)^\ell$, those that lie in the positive orthant $\RR^\ell_+$ are in bijective correspondence with the isolated $\G$-invariant Einstein metrics with $\lambda=1$ on the compact homogeneous space $\G/\mathsf H$.
\end{proof}

\subsection{Newton polytopes}\label{sub:polytopes}
Mixed volumes are usually difficult to compute, but under some conditions, ${\rm MV}(P_1, \ldots, P_\ell) = \ell! \operatorname{Vol}(P)$ where $P = P_1 \cup \cdots \cup P_\ell$ is the union of the polytopes:

\begin{theorem}[{\cite[Cor.~3.7]{FS}}]\label{thm:mv-to-vol}
  Let $P_1, \ldots, P_\ell$ be polytopes in $\RR^\ell$ that are contained in an $\ell$-dimensional polytope $P$.
  Then ${\rm MV}(P_1, \ldots, P_\ell) = \ell! \operatorname{Vol}(P)$ if and only if every proper $t$-dimensional face of $P$ has a nonempty intersection with at least $t + 1$ of the polytopes $P_i$. 
\end{theorem}

To apply Theorem~\ref{thm:mv-to-vol}, we need explicit descriptions of the Newton polytopes of \eqref{eq:system}.
Let $P^\ell_i\coloneqq {\rm Newt}(r^\ell_i) \subset \RR^\ell$ be the Newton polytope of the Laurent polynomial $r^\ell_i$ in \eqref{eq:Ricci-entries} for $i \in [\ell]$. Since the $r^\ell_i$ are homogeneous of degree $-1$, the polytopes $P^\ell_i$ have dimension $\ell - 1$ and lie in the affine hyperplane of points whose coordinates add to $-1$.
We also define $\widetilde{P}^\ell_i \coloneqq {\rm Newt}(f^\ell_i)$ as the Newton polytope of $f^\ell_i$ in \eqref{eq:system}.
The polytope $\widetilde{P}^\ell_i= \conv({\bf 0}, P^\ell_i)$ has dimension $\ell$. Moreover, let
\begin{equation*}
P^\ell \coloneqq P_1^\ell \cup \cdots \cup P_\ell^\ell \quad \text{and }\quad\widetilde P^\ell \coloneqq \widetilde P_1^\ell \cup \cdots \cup \widetilde P_\ell^\ell.  
\end{equation*}
We observe that $P^\ell$ is the Newton polytope of the \emph{scalar curvature} of the homogeneous metric \eqref{eq:diagonal-metric}, which is the Laurent polynomial in $\bf x$ given by
\begin{equation}\label{eq:scalar}
 \operatorname{scal}({\bf x}) \coloneqq \sum_{i=1}^\ell d_i \, r^\ell_i({\bf x})  = \sum_{i=1}^\ell \frac{d_i b_i}{2x_i} - \frac{1}{4} \sum_{i,j,k = 1}^\ell L_{ijk}  \frac{x_k}{x_ix_j},
\end{equation}
see \eqref{eq:Ricci-entries}.
The Newton polytope of $\operatorname{scal}$ was also studied by Graev~\cite{graev1, graev2, graev3} in certain classes of compact homogeneous spaces for which \eqref{eq:scalar} has smaller support.

In order to read the support of the equations in \eqref{eq:system}, we rewrite them without cancellation. For convenience, let $L_{iii}'$ denote the coefficient of $\frac{-1}{4d_ix_i}$ in the $i$th equation, i.e., 
\begin{equation}\label{eq:Lprime}
L_{iii}' \coloneqq L_{iii} + \sum_{j \neq i} 2L_{ijj} - 2b_id_i.
\end{equation}
The system \eqref{eq:system} of homogeneous Einstein equations may then be rewritten as:
\begin{multline}\label{eq:system-nice}
  -4d_i \cdot f^\ell_i({\bf x})  = \frac{L_{iii}'}{x_i} + \sum_{k \in [\ell] \setminus \{i\}} \left(\frac{2L_{iik}x_k}{x_i^2} - \frac{L_{ikk}x_i}{x_k^2}\right)\\ 
  \;+ \sum_{j\neq k \in [\ell] \setminus \{i\}} 2L_{ijk} \left(\frac{x_k}{x_ix_j} + \frac{x_j}{x_ix_k} - \frac{x_i}{x_jx_k}\right) + 4d_i = 0, \quad \quad  1 \leq i \leq \ell, 
\end{multline}
after multiplying through by $-4d_i$. The next lemma is immediate from examining \eqref{eq:system-nice}. 

\begin{lemma}\label{lem:polytopes}
  For all $i \in[ \ell]$, and nonzero ${\bf b}, {\bf d}, L$, the Newton polytopes $\widetilde{P}_i^\ell = {\rm Newt}(f^\ell_i)$ are 
  \begin{align*}
    \widetilde{P}_i^\ell = \conv({\bf 0}, P_i^\ell)\,\, &\textrm{   where    }\,\, P_i^\ell = \conv({\bf e}_i - 2{\bf e}_j, {\bf e}_j - {\bf e}_i - {\bf e}_k \colon j, k \in [\ell]), \,\, \textrm{   and    }\\
    \widetilde{P}^\ell = \conv({\bf 0}, P^\ell) \,\, &\textrm{   where    } \,\, P^\ell = \conv({\bf e}_k - 2{\bf e}_j \colon j,k \in [\ell]).
  \end{align*}
\end{lemma}

\begin{example}
The Newton polygons $P_1^3$, $P_2^3$, $P_3^3$, and $P^3$ for $\ell = 3$ are shown in Figure~\ref{fig:newton}.
\begin{figure}[!htp]
\begin{tikzpicture}[scale=1]
  \def\a{4}
  \def\t{0.25}
  \coordinate (A) at (0,0);
  \coordinate (B) at (\a,0);
  \coordinate (C) at ({0.5*\a},{0.8660254*\a}); 
  \coordinate (T1) at ($(A)!\t!(B)$);
  \coordinate (T6) at ($(A)!\t!(C)$);
  \coordinate (T2) at ($(B)!2*\t!(A)$);
  \coordinate (T3) at ($(B)!\t!(C)$);
  \coordinate (T4) at ($(C)!\t!(B)$);
  \coordinate (T5) at ($(C)!2*\t!(A)$);
  \draw[thick,fill=red!40]
  (T1) -- (T2) -- (T3) -- (T4) -- (T5) -- (T6) -- cycle;
  \node[left] at (T1) {\small ${\bf e}_3 - 2{\bf e}_1$};
  \node at (T1) [circle,fill,inner sep=1pt]{};
  \node[left] at (T6) {\small ${\bf e}_2 - 2{\bf e}_1$};
  \node at (T6) [circle,fill,inner sep=1pt]{};
  \node[right] at (T2) {\small $\;\;{\bf e}_3 - {\bf e}_1 - {\bf e}_2$};
  \node at (T2) [circle,fill,inner sep=1pt]{};
  \node[right] at (T3) {\small ${\bf e}_1 - 2{\bf e}_2$};  
  \node at (T3) [circle,fill,inner sep=1pt]{};
  \node[right] at (T4) {\small ${\bf e}_1 - 2{\bf e}_3$};
  \node at (T4) [circle,fill,inner sep=1pt]{};
  \node[left] at (T5) {\small ${\bf e}_2 - {\bf e}_1 - {\bf e}_3$};
  \node at (T5) [circle,fill,inner sep=1pt]{};
  \node at (2, 1.3) {$P^3_1$};
\end{tikzpicture}\hspace{1 cm}
\begin{tikzpicture}[scale=1]
  \def\a{4}
  \def\t{0.25}
  \coordinate (A) at (0,0);
  \coordinate (B) at (\a,0);
  \coordinate (C) at ({0.5*\a},{0.8660254*\a}); 
  \coordinate (T1) at ($(A)!2*\t!(B)$);
  \coordinate (T6) at ($(A)!\t!(C)$);
  \coordinate (T2) at ($(B)!\t!(A)$);
  \coordinate (T3) at ($(B)!\t!(C)$);
  \coordinate (T4) at ($(C)!2*\t!(B)$);
  \coordinate (T5) at ($(C)!\t!(A)$);
  \draw[thick,fill=green!40]
  (T1) -- (T2) -- (T3) -- (T4) -- (T5) -- (T6) -- cycle;
  \node[left] at (T1) {\small ${\bf e}_3 - {\bf e}_1 - {\bf e}_2\;\;$};
  \node at (T1) [circle,fill,inner sep=1pt]{};
  \node[left] at (T6) {\small ${\bf e}_2 - 2{\bf e}_1$};
  \node at (T6) [circle,fill,inner sep=1pt]{};
  \node[right] at (T2) {\small ${\bf e}_3 - 2{\bf e}_2$};
  \node at (T2) [circle,fill,inner sep=1pt]{};
  \node[right] at (T3) {\small ${\bf e}_1 - 2{\bf e}_2$};  
  \node at (T3) [circle,fill,inner sep=1pt]{};
  \node[right] at (T4) {\small ${\bf e}_1 - {\bf e}_2 - {\bf e}_3$};
  \node at (T4) [circle,fill,inner sep=1pt]{};
  \node[left] at (T5) {\small ${\bf e}_2 - 2{\bf e}_3$};
  \node at (T5) [circle,fill,inner sep=1pt]{};
  \node at (2, 1.3) {$P^3_2$};  
\end{tikzpicture} \vspace{0.75 cm}\\
\begin{tikzpicture}[scale=1]
  \def\a{4}
  \def\t{0.25}
  \coordinate (A) at (0,0);
  \coordinate (B) at (\a,0);
  \coordinate (C) at ({0.5*\a},{0.8660254*\a}); 
  \coordinate (T1) at ($(A)!\t!(B)$);
  \coordinate (T6) at ($(A)!2*\t!(C)$);
  \coordinate (T2) at ($(B)!\t!(A)$);
  \coordinate (T3) at ($(B)!2*\t!(C)$);
  \coordinate (T4) at ($(C)!\t!(B)$);
  \coordinate (T5) at ($(C)!\t!(A)$);
  \draw[thick,fill=blue!40]
  (T1) -- (T2) -- (T3) -- (T4) -- (T5) -- (T6) -- cycle;
  \node[below] at (T1) {\small ${\bf e}_3 - 2{\bf e}_1$};
  \node at (T1) [circle,fill,inner sep=1pt]{};
  \node[left] at (T6) {\small ${\bf e}_2 - {\bf e}_1 - {\bf e}_3$};
  \node at (T6) [circle,fill,inner sep=1pt]{};
  \node[below] at (T2) {\small ${\bf e}_3 - 2{\bf e}_2$};
  \node at (T2) [circle,fill,inner sep=1pt]{};
  \node[right] at (T3) {\small ${\bf e}_1 - {\bf e}_2 - {\bf e}_3$};  
  \node at (T3) [circle,fill,inner sep=1pt]{};
  \node[right] at (T4) {\small ${\bf e}_1 - 2{\bf e}_3$};
  \node at (T4) [circle,fill,inner sep=1pt]{};
  \node[left] at (T5) {\small ${\bf e}_2 - 2{\bf e}_3$};
  \node at (T5) [circle,fill,inner sep=1pt]{};
  \node at (2, 1.3) {$P^3_3$};  
\end{tikzpicture} \hspace{1 cm}
\begin{tikzpicture}[scale=1]
  \def\a{4}
  \def\t{0.25}
  \coordinate (A) at (0,0);
  \coordinate (B) at (\a,0);
  \coordinate (C) at ({0.5*\a},{0.8660254*\a}); 

  \begin{scope}[opacity=.8, transparency group]
    \coordinate (T1) at ($(A)!\t!(B)$);
  \coordinate (T6) at ($(A)!\t!(C)$);
  \coordinate (T2) at ($(B)!2*\t!(A)$);
  \coordinate (T3) at ($(B)!\t!(C)$);
  \coordinate (T4) at ($(C)!\t!(B)$);
  \coordinate (T5) at ($(C)!2*\t!(A)$);
  \draw[dashed,fill=red!40]
  (T1) -- (T2) -- (T3) -- (T4) -- (T5) -- (T6) -- cycle;
  \node[left] at (T1) {};
  \node[left] at (T6) {};
  \node[right] at (T2) {};
  \node[right] at (T3) {};  
  \node[right] at (T4) {};
  \node[left] at (T5) {};
\end{scope}
\begin{scope}[opacity=.8, transparency group]
  \coordinate (T1) at ($(A)!2*\t!(B)$);
  \coordinate (T6) at ($(A)!\t!(C)$);
  \coordinate (T2) at ($(B)!\t!(A)$);
  \coordinate (T3) at ($(B)!\t!(C)$);
  \coordinate (T4) at ($(C)!2*\t!(B)$);
  \coordinate (T5) at ($(C)!\t!(A)$);
   \draw[dashed,fill=green!40]
  (T1) -- (T2) -- (T3) -- (T4) -- (T5) -- (T6) -- cycle;
  \node[left] at (T1) {};
  \node[left] at (T6) {};
  \node[right] at (T2) {};
  \node[right] at (T3) {};  
  \node[right] at (T4) {};
  \node[left] at (T5) {};
  \end{scope}
  \begin{scope}[opacity=.8, transparency group]
    \coordinate (T1) at ($(A)!\t!(B)$);
  \coordinate (T6) at ($(A)!2*\t!(C)$);
  \coordinate (T2) at ($(B)!\t!(A)$);
  \coordinate (T3) at ($(B)!2*\t!(C)$);
  \coordinate (T4) at ($(C)!\t!(B)$);
  \coordinate (T5) at ($(C)!\t!(A)$);
  \draw[dashed,fill=blue!40]
  (T1) -- (T2) -- (T3) -- (T4) -- (T5) -- (T6) -- cycle;
  \node[below] at (T1) {};
  \node[left] at (T6)  {};
  \node[below] at (T2) {};
  \node[right] at (T3) {};  
  \node[right] at (T4) {};
  \node[left] at (T5) {};
  \end{scope}

  \coordinate (T1) at ($(A)!\t!(B)$);
  \coordinate (T6) at ($(A)!\t!(C)$);
  \coordinate (T2) at ($(B)!\t!(A)$);
  \coordinate (T3) at ($(B)!\t!(C)$);
  \coordinate (T4) at ($(C)!\t!(B)$);
  \coordinate (T5) at ($(C)!\t!(A)$);
  \draw[thick]
  (T1) -- (T2) -- (T3) -- (T4) -- (T5) -- (T6) -- cycle;
  \node[below] at (T1) {\small ${\bf e}_3 - 2{\bf e}_1$};
  \node at (T1) [circle,fill,inner sep=1pt]{};
  \node[left] at (T6) {\small ${\bf e}_2 - 2{\bf e}_1$};
  \node at (T6) [circle,fill,inner sep=1pt]{};
  \node[below] at (T2) {\small ${\bf e}_3 - 2{\bf e}_2$};
  \node at (T2) [circle,fill,inner sep=1pt]{};
  \node[right] at (T3) {\small ${\bf e}_1 - 2{\bf e}_2$};  
  \node at (T3) [circle,fill,inner sep=1pt]{};
  \node[right] at (T4) {\small ${\bf e}_1 - 2{\bf e}_3$};
  \node at (T4) [circle,fill,inner sep=1pt]{};
  \node[left] at (T5) {\small ${\bf e}_2 - 2{\bf e}_3$};
  \node at (T5) [circle,fill,inner sep=1pt]{};
\node at (2, 1.3) {$P^3$};  
\end{tikzpicture}
\caption{Newton polygons $P^3_1$, $P^3_2$, $P^3_3$, and $P^3=P^3_1\cup P^3_2\cup P^3_3$, which are contained in the affine plane $\{ {\bf v}\in \RR^3 : v_1+v_2+v_3=-1\}$. The Newton polytope $\widetilde P_i^3$ of $f^3_i$ is the convex hull of $P_i^3$ and the origin ${\bf 0}\in\RR^3$.}
\label{fig:newton}
  \end{figure}
\end{example}

Recall that the \emph{faces} of a polytope $P \subset \mathds{R}^\ell$ are the subsets
\begin{equation} \label{eq:face}
  F_{\bf a}(P) = \big\{{\bf p} \in P : \langle {\bf a} , {\bf p}\rangle \leq \langle {\bf a} , {\bf p}'\rangle \textrm{ for all } {\bf p}' \in P \big\},
\end{equation}
where ${\bf a}$ ranges over all vectors in  $\mathds{R}^\ell$. 
We now explicitly describe the faces of $P^\ell$.

\begin{lemma}\label{lem:Pfaces}
  The proper faces of $P^\ell$ are given by
  \begin{equation*}
    F_{S, T} = \conv({\bf e}_s - 2{\bf e}_t \colon s \in S, \, t \in T)
  \end{equation*}
  where $S, T \subset [\ell]$ are nonempty and disjoint. 
  Furthermore, $\dim(F_{S, T}) = \#(S \cup T) - 2$.
\end{lemma}

\begin{proof}
  Let ${\bf a} = (a_i) \in \RR^\ell$ and let
  \begin{equation}\label{eq:STfaces}
    S = \{i \in [\ell] \colon {a}_i \leq { a}_j \textrm{ for all } j \in [\ell]\}, \qquad
    T = \{i \in [\ell] \colon {a}_i \geq { a}_j \textrm{ for all } j \in [\ell]\}.
  \end{equation}
  We argue that the face $F_{{\bf a}}({P}^\ell)$ is equal to $F_{S, T}$.
  Consider the inner product 
  \begin{equation} \label{eq:inner-prod}
    \left \langle {\bf a}, \sum_{\substack{i,j = 1\\ i \neq j}}^\ell \lambda_{ij} ({\bf e}_i - 2{\bf e}_j) \right \rangle = \sum_{\substack{i,j = 1\\ i \neq j}}^\ell \lambda_{ij} ({a}_i - 2{a}_j)
  \end{equation}
  of ${\bf a}$ with a point in ${P}^\ell$,
  where $0 \leq \lambda_{ij} \leq 1$ and $\sum_{i,j=1,\,i\neq j}^\ell \lambda_{ij} = 1$.
  The quantity ${a}_i - 2{ a}_j$ is minimal when $i \in S$ and $j \in T$. 
  Thus \eqref{eq:inner-prod} is minimized when $\lambda_{ij} = 0$ for all $i,j$ with $i \notin S$ or $j \notin T$.
  In other words, \eqref{eq:inner-prod} is minimized precisely on points in $F_{S,T}$, so $F_{{\bf a}}({P}^\ell) = F_{S, T}$.
  Conversely, given sets $S$, $T$, one may always choose ${\bf a}$ such that \eqref{eq:STfaces} holds. Thus, $F_{S,T}$ is a face for every nonempty and disjoint $S, T\subset [\ell]$.

  The claim that the face $F_{S, T}$ has dimension $\#(S \cup T) - 2$ is equivalent to proving that the cone in the normal fan corresponding to $F_{S, T}$ has dimension $\ell - (\#(S \cup T) - 2) = \ell - \#(S \cup T) + 2$.
  In the first part of the proof, we found that a vector ${\bf a}$ is normal to $F_{S, T}$ precisely when the minima of ${\bf a}$ have indices $S$ and the maxima have indices $T$. 
  To generate such a vector $\bf a$, we have $\ell - \#(S \cup T) + 2$ degrees of freedom: we may choose values for the $\ell - \#(S \cup T)$ entries of $\bf a$ not having indices in $S \cup T$ and then we choose a single value for all entries with index in $S$ and another single value for all entries with index in $T$.
  The remaining constraints placed on the entries of $\bf a$ are inequalities, so they do not decrease the dimension. 
\end{proof}

\begin{example}\label{ex:short-long}
  Consider the polygon $P^3$, shown in Figure~\ref{fig:newton}.
  The long edges correspond to the partitions $(S, T) = (1,23),\, (2,13),\, (3,12)$ and the short edges correspond to the partitions
  $(S, T) = (23,1),\, (13,2),\, (12,3)$.
  The $6$ vertices are given by all choices of $(S,T)$ with $\#S = \#T = 1$. 
\end{example}

We now prove a related lemma about the faces of the $P_i^\ell$. 
\begin{lemma}\label{lem:Pifaces}
  Let ${\bf a} \in \mathds{R}^\ell$ and $S,T$ be as in \eqref{eq:STfaces}.
  If $i \in S \cup T$, then $F_{\bf a}(P_i^\ell) = P_i^\ell \cap F_{S,T}$. 
\end{lemma}
\begin{proof}
  By Lemma~\ref{lem:polytopes}, the inner product of $\bf a$ with a point in $P_i^\ell$ is
  \begin{equation}\label{eq:Pi-inner-prod}
    \left \langle\! {\bf a}, \sum_{\substack{j=1\\j\neq i}}^\ell \lambda_{jj} ({\bf e}_i - 2{\bf e}_j) + \! \sum_{\substack{j,k=1\\j\neq i,k}}^\ell\!\! \lambda_{jk} ({\bf e}_j - {\bf e}_k - {\bf e}_i)\!  \right \rangle
    \! =\! \sum_{\substack{j=1\\j\neq i}}^\ell \!\lambda_{jj} (a_i - 2a_j) +\! \sum_{\substack{j,k=1\\j\neq i,k}}^\ell\!\! \lambda_{jk} (a_j - a_k - a_i)
  \end{equation}
  where $0 \leq \lambda_{jk} \leq 1$ and $\sum_{j,k =1}^\ell \lambda_{jk}= 1$.
  If $i \in S$, then \eqref{eq:Pi-inner-prod} is minimized when $\lambda_{jk} > 0$ if and only if $j = k \in T$.
  Then $F_{\bf a}(P_i^\ell) = {\rm conv}({\bf e}_i - 2{\bf e}_j: j \in T) = P_i^\ell \cap F_{S,T}$.
  If $i \in T$, then \eqref{eq:Pi-inner-prod} is minimized when $\lambda_{jk} > 0$ if and only if $j \in S$ and $k \in T$.
  Then $F_{\bf a}(P_i^\ell) = {\rm conv}({\bf e}_j - {\bf e}_k - {\bf e}_i: j \in S, k \in T) = P_i^\ell \cap F_{S,T}$.
  The conclusion follows. 
\end{proof}

We are now ready to prove the main result of this section.

\begin{proof}[Proof of Theorem~\ref{thm:mv-vol}.]
  By Theorem~\ref{thm:mv-to-vol}, it suffices to show that each proper $t$-dimensional face of $\widetilde P^\ell$ intersects at least $t+1$ of the $\widetilde{P}_1^\ell, \ldots, \widetilde{P}_\ell^\ell$ nontrivially.
  Note that the faces ${\bf 0}$ and $P^\ell$ of $\widetilde{P}^\ell$ intersect $\widetilde P_i^\ell$ for all $i \in [\ell]$. Let $F$ be a proper face of $\widetilde P^\ell$. If $F \subset P^\ell$, then $F = F_{S,T}$ for sets $S,T$ as in Lemma~\ref{lem:Pfaces}. Otherwise, there exist $S,T$ such that $F = \operatorname{conv}(F_{S,T}, {\bf 0})$.
  Thus $\dim(F) \in \{\#(S \cup T) - 2, \#(S \cup T) -1 \}$.
  The face $F_{S,T}$ intersects the polytopes $\{P_i^\ell\}_{i \in S \cup T}$ by Lemma~\ref{lem:Pifaces}.
  Since $P_i^\ell \subset \widetilde P_i^\ell$, the intersection $F \cap \widetilde P_i^\ell \supseteq F_{S,T} \cap P_i^\ell$ is nonempty. Hence $F$ intersects $\#(S \cup T)$ polytopes $\Tilde{P}_i^\ell$. The claim follows from Theorem~\ref{thm:mv-to-vol}.
\end{proof}

\subsection{Volume of a permutohedron}\label{sub:volume}
Given a vector ${\bf y} \in \RR^\ell$ with $y_1 \geq y_2 \geq \cdots \geq y_\ell$, let $P({\bf y})$ be the \emph{permutohedron} defined as the convex hull of all permutations of $\bf y$. 
Since $P({\bf y})$ has dimension at most $\ell - 1$, its $\ell$-dimensional volume is zero.
As in \cite{postnikov}, we use the $(\ell - 1)$-dimensional volume after projecting away from the last coordinate; for clarity, we denote this ``projected'' volume by $\operatorname{pVol}$. 
We now express the normalized volume of $\widetilde{P}^\ell$ in terms of $\operatorname{pVol}(P^\ell)$, and then prove that $(\ell - 1)!\operatorname{pVol}(P^\ell)= D_{\ell - 1}$.

\begin{lemma}\label{lem:nvol}
  Let $P \subset \RR^\ell$ be a polytope contained in the hyperplane $\{y_1 + \cdots + y_\ell = -1\}$ and $\widetilde P = \conv({\bf 0}, P)$. 
  Then $(\ell - 1)! \operatorname{pVol}(P) =  \ell!  \operatorname{Vol}(\widetilde P)$.
\end{lemma}

\begin{proof}
  We can triangulate $P$ and lift to a triangulation of $\widetilde P$, so it suffices to prove the claim for the simplex $P =  \conv(-{\bf e}_1, \ldots, -{\bf e}_\ell)$.
  Let $P' = \conv({\bf 0}, -{\bf e}_1, \ldots, -{\bf e}_{\ell - 1}) \subset \RR^{\ell-1}$ be the projection of $P$ onto the first $\ell -1$ coordinates. 
  A direct computation shows that $\ell!  \operatorname{Vol}(\widetilde P) = (\ell - 1)!  \operatorname{Vol}(P') = (\ell -1)! \operatorname{pVol}(P)$.
\end{proof}

The normalized volume of a permutohedron was computed by Postnikov~\cite{postnikov}:

\begin{theorem}[{\cite[Thm.~3.2]{postnikov}}]
  Given ${\bf y} \in \RR^\ell$, the normalized volume of $P({\bf y})$ is
  \begin{equation}\label{eq:postnikov}
    (\ell-1)!\operatorname{pVol}(P({\bf y})) = \sum_{\mathbf{c}} (-1)^{|I_{\mathbf c}|} \, {\rm des}_\ell(I_\mathbf{c}) \, \binom{\ell-1}{\mathbf{c}} {\bf y}^\mathbf{c},
  \end{equation}
  where the sum is over $\mathbf{c} = (c_1, \ldots, c_\ell) \in \mathds{N}^\ell$ such that $\sum_{i = 1}^\ell c_i = \ell-1$.
\end{theorem}

Let us explain the notation used above. First, $\binom{\ell-1}{\mathbf c}$ denotes the multinomial coefficient $\binom{\ell-1}{c_1, \ldots, c_\ell} = \frac{(\ell - 1)!}{c_1! \cdots c_\ell!}$. Second, recall that the {\em descent set} of a permutation $\sigma$ on $[\ell]$ is the set of positions of descents, i.e., $\{i \in [\ell-1] : \sigma(i) > \sigma(i+1)\}$. Then, given a subset $S \subseteq [\ell - 1]$, the number of permutations on $[\ell]$ with descent set $S$ is denoted ${\rm des}_\ell(S)$.

Finally, to define $I_\mathbf{c}$, we construct a lattice path from $(0,0)$ to $(\ell-1,\ell-1)$ such that the $(i-1)$st column has exactly $c_{i}$ steps up. We then divide the diagonal path from $(0, 0)$ to $(\ell-1,\ell-1)$ into $\ell-1$ segments with integer endpoints. We label these segments by the larger endpoint, so the segment from $(0,0)$ to $(1,1)$ is labeled $1$. 
Then $I_{\bf c}$ is the set of these diagonal segments which lie above the path.

For example, the vector ${\bf c} = (0,0,2,2,2,0,0,1,2,0,1)\in \mathds N^\ell$, with $\ell=11$, yields the path
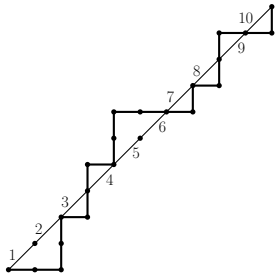
\begin{equation*}
  \begin{tikzpicture}[scale = 0.35, transform shape]
    \draw (0,0) -- (10, 10);
    \draw[thick] (0,0) -- (2,0) -- (2,2) -- (3,2) -- (3,4) -- (4,4) -- (4,6) -- (6,6) --(7, 6) -- (7,7) -- (8,7) -- (8,9) -- (10, 9) -- (10,10);
    \node at (0,0) [circle,fill,inner sep=2pt]{};
    \node at (1,0) [circle,fill,inner sep=2pt]{};
    \node at (1,1) [circle,fill,inner sep=2pt]{};
    \node at (2,0) [circle,fill,inner sep=2pt]{};
    \node at (2,2) [circle,fill,inner sep=2pt]{};
    \node at (2,1) [circle,fill,inner sep=2pt]{};
    \node at (3,2) [circle,fill,inner sep=2pt]{};
    \node at (3,3) [circle,fill,inner sep=2pt]{};
    \node at (3,4) [circle,fill,inner sep=2pt]{};
    \node at (4,4) [circle,fill,inner sep=2pt]{};
    \node at (4,5) [circle,fill,inner sep=2pt]{};
    \node at (4,6) [circle,fill,inner sep=2pt]{};
    \node at (5,5) [circle,fill,inner sep=2pt]{};
    \node at (5,6) [circle,fill,inner sep=2pt]{};
    \node at (6,6) [circle,fill,inner sep=2pt]{};
    \node at (7,6) [circle,fill,inner sep=2pt]{};
    \node at (7,7) [circle,fill,inner sep=2pt]{};
    \node at (8,7) [circle,fill,inner sep=2pt]{};
    \node at (8,8) [circle,fill,inner sep=2pt]{};
    \node at (8,9) [circle,fill,inner sep=2pt]{};
    \node at (9,9) [circle,fill,inner sep=2pt]{};
    \node at (10,9) [circle,fill,inner sep=2pt]{};
    \node at (10,10) [circle,fill,inner sep=2pt]{};
    \node[above=2pt, left=2pt, font=\LARGE] at (0.5,0.5) {1};
    \node[above=2pt, left=2pt, font=\LARGE] at (1.5,1.5) {2}; 
    \node[above=2pt, left=2pt, font=\LARGE] at (2.5,2.5) {3};
    \node[below=2pt, right=2pt, font=\LARGE] at (3.5,3.5) {4};
    \node[below=2pt, right=2pt, font=\LARGE] at (4.5,4.5) {5};
    \node[below=2pt, right=2pt, font=\LARGE] at (5.5,5.5) {6};
    \node[above=2pt, left=2pt, font=\LARGE] at (6.5,6.5) {7};
    \node[above=2pt, left=2pt, font=\LARGE] at (7.5,7.5) {8};
    \node[below=2pt, right=2pt, font=\LARGE] at (8.5,8.5) {9};
    \node[above=2pt, left=2pt, font=\LARGE] at (9.5,9.5) {10};
  \end{tikzpicture}
\end{equation*}
from which we observe that the segments above the path are $I_{\bf c} = \{1,2,3,7,8,10\}$. 

We now prove the main result of this section, concluding the proof of Theorem~\ref{thm:delannoy}.

\begin{proof}[Proof of Theorem~\ref{thm:centraldelannoy}]
By Lemma~\ref{lem:nvol}, it suffices to compute  $(\ell - 1)!\operatorname{pVol}({P}^\ell)$. We write ${P}^\ell = P({\bf y})$ with ${\bf y} = (1, 0, 0, \ldots, 0, -2)$. Note that if any of $c_2, \ldots, c_{\ell-1}$ is nonzero, then ${\bf y}^{\bf c} = 0$, so it suffices to sum over $\bf c$ with $c_2 = \cdots = c_{\ell-1} = 0$. 
  Thus, \eqref{eq:postnikov} simplifies to
  \begin{multline*}
    (\ell - 1)! \operatorname{Vol}(\Tilde{P}^\ell) \; = \sum_{c_1 + c_\ell= \ell-1} (-1)^{|I_{\mathbf c}|} \, {\rm des}_\ell(I_{\mathbf c}) \, \binom{\ell - 1}{c_1, c_\ell} (1)^{c_1}(-2)^{c_\ell}  \\
 =\; \sum_{c_\ell = 0}^{\ell-1} (-1)^{|I_{\mathbf c}|} \, {\rm des}_\ell(I_{\mathbf c}) \, \binom{\ell - 1}{c_\ell} (-2)^{c_\ell}
  \end{multline*}
  where $\mathbf{c} = (\ell - 1 - c_\ell, 0, \ldots, 0, c_\ell)$.
  This choice of $\mathbf c$ corresponds to the path with $\ell - 1 - c_\ell$ vertical steps in the $0$th column, $c_\ell$ vertical steps in the $(\ell - 1)$st column, and $0$ vertical steps in the other columns:
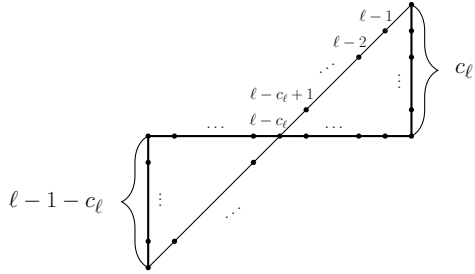
\begin{equation*}
  \begin{tikzpicture}[scale = 0.35, transform shape]
    \draw (0,0) -- (10,10);
    \draw[thick] (0,5) -- (10, 5);
    \draw[thick] (0,0) -- (0,5);
    \draw[thick] (10,5) -- (10,10);
    \node at (0,0) [circle,fill,inner sep=2pt]{};
    \node at (0,5) [circle,fill,inner sep=2pt]{};
    \node at (0,1) [circle,fill,inner sep=2pt]{};
    \node at (1,1) [circle,fill,inner sep=2pt]{};
    \node at (0,4) [circle,fill,inner sep=2pt]{};
    \node at (1,5) [circle,fill,inner sep=2pt]{};
    \node at (4,4) [circle,fill,inner sep=2pt]{};
    \node at (4,5) [circle,fill,inner sep=2pt]{};
    \node at (5,5) [circle,fill,inner sep=2pt]{};
    \node at (6,6) [circle,fill,inner sep=2pt]{};
    \node at (6,5) [circle,fill,inner sep=2pt]{};
    \node at (8,5) [circle,fill,inner sep=2pt]{};
    \node at (9,5) [circle,fill,inner sep=2pt]{};
    \node at (10,5) [circle,fill,inner sep=2pt]{};
    \node at (10,6) [circle,fill,inner sep=2pt]{};
    \node at (10,8) [circle,fill,inner sep=2pt]{};
    \node at (8,8) [circle,fill,inner sep=2pt]{};
    \node at (9,9) [circle,fill,inner sep=2pt]{};
    \node at (10,9) [circle,fill,inner sep=2pt]{};
    \node at (10,10) [circle,fill,inner sep=2pt]{};
    \node[left=7pt, font=\LARGE] at (10,7.25) {$\vdots$};
    \node[right=7pt, font=\LARGE] at (0,2.75) {$\vdots$};
    \node[below=7pt, right=7pt, font=\LARGE] at (2.5,2.5) {$\iddots$};
    \node[above=10pt, left=10pt, font=\LARGE] at (3.5,5) {$\dots$};
    \node[above=7pt, left=7pt, font=\LARGE] at (7.5,7.5) {$\iddots$};
    \node[above=10pt, left=10pt, font=\LARGE] at (8,5) {$\dots$};
    \draw[decorate,decoration={brace,amplitude=10pt,raise=0pt},yshift=0pt]
    (0,0) -- (0,5) node [black,midway,xshift=-3.5cm, font=\Huge] {$\ell - 1- c_\ell$};
    \draw[decorate,decoration={brace,amplitude=10pt, mirror, raise=0pt},yshift=0pt]
    (10,5) -- (10,10) node [black,midway,xshift=2cm, font=\Huge] {$c_\ell$};
    \node[above=2pt, left=2pt, font=\LARGE] at (5.5,5.5) {$\ell - c_\ell$};
    \node[above=2pt, left=2pt, font=\LARGE] at (6.5,6.5) {$\ell - c_\ell + 1$};
    \node[above=2pt, left=2pt, font=\LARGE] at (8.5,8.5) {$\ell - 2$};
    \node[above=2pt, left=2pt, font=\LARGE] at (9.5,9.5) {$\ell - 1$};    
\end{tikzpicture}
\end{equation*}
  We therefore have $I_\mathbf{c} = \{\ell - c_\ell, \ldots, \ell - 1\}$, so $|I_\mathbf{c}| = c_\ell$ and $$(\ell - 1)! \operatorname{Vol}(\Tilde{P}^\ell) = \sum_{c_\ell = 0}^{\ell-1} 2^{c_\ell} {\rm des}_\ell(I_{\mathbf c}) \binom{\ell - 1}{c_\ell}.$$
  Finally, we compute the number of permutations on $[\ell]$ which have descents precisely in the last $c_\ell$ entries.
  To construct such a permutation, we must place $\ell$ in the $\ell - c_\ell$ position to avoid creating an unwanted ascent or descent.
  Then any choice of $c_\ell$ elements of $[\ell-1]$ gives a permutation: place the $c_\ell$ elements in decreasing order to the right of $\ell$  and place the remaining numbers to the left of $\ell$ in increasing order.
  Thus, ${\rm des}_\ell(I_{\mathbf c}) = \binom{\ell - 1}{c_\ell}$ and we have
  \begin{equation*}
    \ell!  \operatorname{Vol}(\widetilde P^\ell) =
    (\ell - 1)! \operatorname{pVol}({P}^\ell)= 
    \sum_{k = 0}^{\ell - 1}
    2^k \binom{\ell - 1}{k}^2 = D_{\ell - 1}.\qedhere
  \end{equation*}
\end{proof}

\section{Connections to Algebraic Statistics}\label{sec:algstat}
We now explain how to interpret the system \eqref{eq:system}, equivalently \eqref{eq:system-nice}, in the context of algebraic statistics. This perspective plays a crucial role in the proof of Theorem~\ref{thm:generic-finiteness}.
Our first step is to realize \eqref{eq:system-nice} as the critical equations of a maximum likelihood estimation problem on a scaled toric variety.

\subsection{Maximum likelihood estimation and scaled toric varieties}
Given data ${\bf u} \in \mathds{C}^r$, the \emph{maximum likelihood estimation} problem on a complex projective variety $V \subset  \mathds P_{\CC}^{r-1}$ is
\begin{equation}\label{eq:optimization}
  \textrm{maximize} \quad \sum_{i=1}^r u_{i} \log(p_i) - u_+ \log(p_+) \quad \textrm{ subject to } \quad (p_1 : \cdots : p_r) \in V,
\end{equation}
where $u_+ \coloneqq u_1 + \cdots + u_r$ and $p_+ \coloneqq p_1 + \cdots + p_r$. 
Because the derivatives of the objective function in \eqref{eq:optimization} are algebraic functions, one can study the set of critical points of \eqref{eq:optimization} using algebraic geometry; this is the perspective taken by algebraic statistics; see \cite[Chap.~7]{sullivant}.
For generic data $\bf u$, the optimization problem \eqref{eq:optimization} has finitely many critical points
and the number of critical points is independent of the choice of $\bf u$.
This number is called the \emph{maximum likelihood (ML) degree} of the variety; see \cite{CHKS06}.

A vector ${\bf c} \in (\mathds{C}^*)^r$ and a matrix $A=({\bf a}_1 \, \cdots \, {\bf a}_r) \in \mathds{Z}^{\ell \times r}$ define a \emph{scaled toric variety} $V_{A, \bf c} \subset \mathds{P}_\CC^{r-1}$ as the Zariski closure of the image of the scaled monomial map
\begin{equation}\label{eq:parametrization}
  (\CC^*)^\ell \to \mathds{P}_\CC^{r-1}, \qquad {\bf x} \mapsto (c_1{\bf x}^{{\bf a}_1}: \cdots: c_r{\bf x}^{{\bf a}_r}).
\end{equation}
The ML degree of scaled toric varieties was studied in \cite{ABB+}.
By \cite[Prop.~6]{ABB+}, a point $\bf p$ is a critical point of \eqref{eq:optimization} if it satisfies the following \emph{critical equations}:
\begin{equation}\label{eq:crit-eqns}
\phantom{ \qquad  {\bf p} \in V_{A, {\bf L}}.}
  u_+ \cdot A\cdot {\bf p} = p_+ \cdot A\cdot {\bf u}, \qquad  {\bf p} \in V_{A, {\bf c}}.
\end{equation}

\subsection{Reinterpreting the homogeneous Einstein equations}\label{sec:reinterpret}
We now write \eqref{eq:system-nice} in the form \eqref{eq:crit-eqns}.
We begin by identifying a matrix $A$ and a vector $\bf L$ such that \eqref{eq:system-nice} factors as the matrix equation
\begin{equation}\label{eq:system-matrix}
  A \cdot \operatorname{diag}({\bf L}) \cdot {\bf x}^A = 4{\bf d}.
\end{equation}

\begin{example}\label{ex:Asmalll}
  For $\ell = 2$, the system \eqref{eq:system-nice} can be written in the form \eqref{eq:system-matrix} as follows:
  \begin{equation*}
    \begin{bmatrix}
      -2 & 1 & -1 & 0\\
      1 & -2 & 0 & -1
    \end{bmatrix}
    \begin{bmatrix}
      L_{112}\\
      & L_{122} \\
      & & L_{111}'\\
      & & & L_{222}'
    \end{bmatrix}
    \begin{bmatrix}
      x_2/x_1^2\\
      x_1/x_2^2\\
      1/x_1\\
      1/x_2
    \end{bmatrix}
    = 
    \begin{bmatrix}
      4d_1 \\ 4d_2
    \end{bmatrix}.
  \end{equation*}
  For $\ell = 3$, the system \eqref{eq:system-nice} can be written in the form  \eqref{eq:system-matrix} with
  \setcounter{MaxMatrixCols}{20}
\[\begingroup\BigColSep
  A =
  \begin{bmatrix}
    -2 & 1 & -2 & 1 & 0 & 0 & 1 & -1 & -1 & -1 & 0 & 0 \\
    1 & -2 & 0 & 0 & -2 & 1 & -1 & 1 & -1 & 0 & -1 & 0 \\
    0 & 0 & 1 & -2 & 1 & -2 & -1 & -1 & 1 & 0 & 0 & -1
  \end{bmatrix},
    \endgroup
\]
  \[{\bf L} =
  \begin{bmatrix}
    L_{112} & L_{122} & L_{113} & L_{133} & L_{223} & L_{233}
    & 2L_{123} & 2L_{123} & 2L_{123} & L_{111}' & L_{222}' & L_{333}'
  \end{bmatrix}^T,
\]
\[\begingroup\BigColSep
  {\bf x}^A =
  \begin{bmatrix}
    \frac{x_2}{x_1^2} & 
    \frac{x_1}{x_2^2} &
    \frac{x_3}{x_1^2} &
    \frac{x_1}{x_3^2} &
    \frac{x_3}{x_2^2} &
    \frac{x_2}{x_3^2} &
    \frac{x_1}{x_2x_3} &
    \frac{x_2}{x_1x_3} &
    \frac{x_3}{x_1x_2} &
    \frac{1}{x_1} &
    \frac{1}{x_2} &
    \frac{1}{x_3}
  \end{bmatrix}^T.\;\;
  \endgroup
\]
\end{example}

For general $\ell\geq3$, let $r = 2\binom{\ell}{2} + 3 \binom{\ell}{3} + \ell$. 
Define $A\in \ZZ^{\ell\times r}$ as the $\ell \times r$ matrix whose first $2\binom{\ell}{2}$ columns are ${\bf e}_i - 2{\bf e}_k$ for $i \neq k \in [\ell]$, whose next $3 \binom{\ell}{3}$ columns are ${\bf e}_i - {\bf e}_j - {\bf e}_k$ for $i,j,k \in [\ell]$ distinct, and whose last $\ell$ columns are $-{\bf e}_i$ for $i \in [\ell]$. Let ${\bf L}\in (\mathds{C}^*)^r$ be the vector whose first $2 \binom{\ell}{2}$ entries are $L_{ikk}$ for $i \neq k \in [\ell]$, whose next $3 \binom{\ell}{3}$ entries are $2L_{ijk}$ for $i,j,k \in [\ell]$ distinct, and whose last $\ell$ entries are $L_{iii}'$ for $i \in [\ell]$.
In the above, the columns of $A$ and the entries of $\bf L$ have to be ordered accordingly, as in \Cref{ex:Asmalll}.

\begin{theorem}\label{thm:rewriteMLE}
  There exists a data vector ${\bf u} \in \mathds R^r$ such that \eqref{eq:system-nice} are the critical equations of \eqref{eq:optimization} on the scaled toric variety $V_{A, {\bf L}}$, where $A$ and $\bf L$ are as defined above. 
\end{theorem}

\begin{proof}
  Since the columns of $A$ span all of $\mathds R^\ell$, there exists a vector ${\bf u} \in \RR^{r}$ such that $4{\bf d} = A{\bf u}$. 
  It follows from \eqref{eq:system-matrix} that, if ${\bf \widetilde{L}} = \frac{1}{u_+} {\bf L}$, then 
  \begin{equation}\label{eq:ml_eqs}
    u_+ A \cdot \operatorname{diag}\big ({\bf \widetilde{L}} \big ) \cdot {\bf x}^A
    = A \cdot \operatorname{diag}\big ({\bf L} \big ) \cdot {\bf x}^A 
    = 4{\bf d}
     = A{\bf u}. 
  \end{equation}
  This is the parametric version of \eqref{eq:crit-eqns} with ${\bf c} = {\bf \widetilde{L}}$; namely, \eqref{eq:crit-eqns} is obtained from \eqref{eq:ml_eqs} by setting ${\bf p} = \operatorname{diag}({\bf \widetilde{L}}) \cdot {\bf x}^A$ and eliminating the variables $x_1, \ldots, x_\ell$.
  Therefore \eqref{eq:ml_eqs}, or equivalently \eqref{eq:system-nice}, are the (parametric) critical equations on $V_{A, {\bf \widetilde{L}}} = V_{A, {\bf L}}$. 
\end{proof}

\begin{remark}
  Birch's Theorem states that if $\bf L$ and $\bf u$ are both positive vectors, then \eqref{eq:crit-eqns} has at most one positive solution ${\bf p} \in V_{A, \bf L}$; see \cite[Thm.~9]{ABB+}. So, it is natural to ask whether this can be applied in the situation of \Cref{thm:rewriteMLE}, as it would prove that there is exactly one homogeneous Einstein metric on the corresponding homogeneous space. 
  However, the positive hull of the columns of $A$ does not intersect the positive orthant, so Birch's Theorem \emph{never} applies here, because the vector $\bf d$ has positive coordinates for all homogeneous spaces. This matches the geometric expectation that, if a compact homogeneous space admits homogeneous Einstein metrics, then they are usually not unique.
\end{remark}

So far, we have required that ${\bf L}$ lies in $(\mathds{C}^*)^r$. In practice, this is not a realistic assumption, as often some structure constants $L_{ijk}$ vanish. In that case, even though some entries of $\bf L$ are zero, \eqref{eq:system-nice} are still the critical equations on a scaled toric variety defined by removing the zero entries of $\bf L$ and the corresponding columns of $A$, provided $\bf d$ is still in the column span of~$A$. 
We use this to prove that \eqref{eq:system-matrix} is BKK generic for generic parameters $\bf L$ and $\bf d$.

\begin{proof}[Proof of Theorem~\ref{thm:generic-finiteness}]
  It suffices to prove that the BKK discriminant does not vanish identically, i.e., that there is some choice of parameters for which the system has $D_{\ell-1}$ solutions.
 In \eqref{eq:system-matrix}, replace $A$ with an $\ell \times 2\binom{\ell}{2}$ matrix $A'$ whose columns are ${\bf e}_i - 2{\bf e}_k$ and ${\bf L}$ with a vector ${\bf L}'$ of length $2\binom{\ell}{2}$ whose entries are $L_{ikk}$; this corresponds to choosing $L$ with $L_{ijk} = L_{iii}' = 0$ for all $i,j,k \in [\ell]$ distinct.
  For example, for $\ell = 3$, this system is
  \begin{equation*}
    \underbrace{\begin{bmatrix}
    -2 & 1 & -2 & 1 & 0 & 0 \\
    1 & -2 & 0 & 0 & -2 & 1 \\
    0 & 0 & 1 & -2 & 1 & -2 
    \end{bmatrix}}_{A'}\!
    \underbrace{
  \begin{bmatrix}
    L_{112}\\
    & L_{122}\\
    & & L_{113}\\
    & & & L_{133}\\
    & & & & L_{223}\\
    & & & & & L_{233}
  \end{bmatrix}}_{{\operatorname{diag}({\bf L}')}}\!
  \begin{bmatrix}
    {x_2}/{x_1^2}\\
    {x_1}/{x_2^2} \\
    {x_3}/{x_1^2} \\
    {x_1}/{x_3^2} \\
    {x_3}/{x_2^2} \\
    {x_2}/{x_3^2}
  \end{bmatrix}
  \!\!=\!\!
  \begin{bmatrix}
    4d_1 \\ 4d_2 \\ 4d_3
  \end{bmatrix},
  \end{equation*}
cf.~Example~\ref{ex:Asmalll}.
  Since $A'$ still has full rank, we can write this system in the form \eqref{eq:ml_eqs}.
  By assumption, the vectors $\bf L'$ and $\bf d$ are generic.
  Therefore, the number of solutions to $A' \cdot \operatorname{diag}({\bf L}') \cdot {\bf x}^{A'} = 4{\bf d}$ is equal to the ML degree of the scaled toric variety $V_{A', {\bf L}'}$. 
  Since $\bf L'$ and $\bf d$ are generic, by \cite[Cor.~8]{ABB+}, this ML degree is equal to the degree of the toric variety $V_{A', \bf 1}$. The degree of toric variety $V_{A', \bf 1}$ is the number of points of the intersection of $V_{A', \bf 1}$ with a linear subspace of $\mathds{P}^{2\binom{\ell}{2} - 1}_{\CC}$ of codimension $(\ell-1)$. By Theorem~\ref{thm:bkk} and the fact that ${\rm MV}(P, \ldots, P)$ is the normalized volume of $P$, this degree is the normalized volume of $\operatorname{conv}(A') = \operatorname{conv}(A)$, which is $D_{\ell-1}$ by Theorem~\ref{thm:centraldelannoy}.
\end{proof}

\subsection{Facial systems and the BKK discriminant}
Given a system of Laurent polynomials $\mathcal F = \{f_1, \ldots, f_\ell\}$,
the \emph{facial systems} of $\mathcal F$ are obtained by restricting the support of each polynomial $f_i ({\bf x})= \sum_{{\bf a'} \in \mathds{Z}^\ell \cap P_i} c_{i,\bf a'} \bf x^{a'}$ to a proper face of its Newton polytope $P_i$; see Section~\ref{subsec:bkk_bound}. Namely, if ${\bf a} \in \RR^\ell$, the facial system $\mathcal F^{\bf a}$ consists of the polynomials $f_{i,{\bf a}} ({\bf x}) = \sum_{{\bf a}' \in \ZZ^\ell \cap F_{\bf a}(P_i)} c_{i,\bf a'} {\bf x}^{\bf a'}$, where $F_{\bf a}(P_i)$ is the face of $P_i$ given by \eqref{eq:face}. The polynomial $f_{i,{\bf a}}$ is called the \emph{restriction of $f_i$ to $F_{\bf a}(P_i)$}. 
Recall that the parameter locus where the BKK bound of $\mathcal F$ is not achieved, i.e., where $\mathcal F$ is not BKK generic, is called the BKK discriminant of $\mathcal F$.
The following theorem describes it in terms of facial systems:

\begin{theorem}[Bernstein~{\cite[Thm.~B]{bernstein}}]\label{thm:other}
  Suppose that for all ${\bf a} \neq {\bf 0} \in \RR^\ell$, the facial system $\mathcal F^{\bf a}$ has no roots in $(\CC^*)^{\ell}$. Then all the roots of the system $\mathcal F$ are isolated and the number of solutions to the system $\mathcal F$ is equal to the mixed volume ${\rm MV}(P_1, \ldots, P_\ell)$.
\end{theorem}

Any system $\mathcal F$ of Laurent polynomials has only finitely many facial systems, so one only needs to check finitely many conditions to apply Theorem~\ref{thm:other}. 

Consider the system \eqref{eq:system-nice} and the face $P^\ell$ of the polytope $\widetilde{P}^\ell$ in Lemma~\ref{lem:polytopes}. The corresponding facial system is $r_1^\ell ({\bf x})= \cdots = r_\ell^\ell({\bf x}) = 0$, where $r_i^\ell$ are the Laurent polynomials in \eqref{eq:Ricci-entries}. 
Note that these are multiples of the toric derivatives of $\operatorname{scal}$, namely,
\begin{equation}\label{eq:ri-toricder}
  -d_i \, r_i^\ell({\bf x}) = x_i \frac{\partial}{\partial x_i} \operatorname{scal}({\bf x}),
\end{equation}
and recall that $\operatorname{scal}$ is homogeneous; see \eqref{eq:scalar}. The set of parameters where the toric derivatives of a homogeneous function have a common solution is the principal $A$-determinant. In general, 
the \emph{principal $A$-determinant} \cite[Chap.~10]{GKZ} of a homogeneous polynomial $f ({\bf x})= \sum_{i=1}^r c_i {\bf x}^{{\bf a}_i}$, where ${\bf c}\in \CC^r$ and $A = ({\bf a}_1 \cdots {\bf a}_r)\in\mathds Z^{\ell \times r}$, is defined as the~$A$-resultant
  \begin{equation}\label{eq:Adet}
    E_A(f) = R_A \left (x_1 \frac{\partial f}{\partial x_1}, \,\dots,\, x_\ell \frac{\partial f}{\partial x_\ell} \right ).
  \end{equation}
By \cite[Thm.~10.1.2]{GKZ}, the principal $A$-determinant of such a polynomial $f$ factors as
\begin{equation}\label{eq:Adet-factors}
  E_A(f) = \prod_{F \textrm{ face of }{\rm Newt}(f)} (\Delta_{F \cap A})^{\alpha_F},
\end{equation}
where $\alpha_F \in \mathds N$ and $\Delta_{F \cap A}$ is an $A$-discriminant.
Namely, if the variety 
\begin{equation*}
  \nabla_{F \cap A} = \overline{\left \{{\bf c} \in \mathds{C}^r : \textrm{there exists ${\bf x} \in (\CC^*)^\ell$ such that $\frac{\partial f_F}{\partial x_i}({\bf x}) = 0$ for $i \in [\ell]$}\right \}}
\end{equation*}
has codimension 1, then the \emph{$A$-discriminant} $\Delta_{F \cap A}$ is its defining polynomial.
If $\nabla_{F \cap A}$ has higher codimension, then $\Delta_{F \cap A}=1$. 
Here $f_F$ denotes the restriction of $f$ to the face $F$. 

We now prove that the zero set of \eqref{eq:bkkdiscriminant} contains the BKK discriminant of \eqref{eq:system-nice}.

\begin{proof}[Proof of Theorem~\ref{thm:bkkdiscriminant}]
  We show that the system \eqref{eq:system-nice}, henceforth denoted $\mathcal F$, satisfies the hypotheses of \Cref{thm:other} and is hence BKK generic, provided that \eqref{eq:bkkdiscriminant} does not vanish.
  We begin by fixing a vector ${\bf a} \in \mathds R^\ell$ and letting $S,T$ be as in \eqref{eq:STfaces}.
  Observe that if $\mathcal F^{\bf a}$ has a solution, then its subsystem $\mathcal F^{S,T} \coloneqq \{f_{i, \bf a}({\bf x}) = 0\}_{i \in S \cup T}$
  also has a solution.
  We will prove that if $\mathcal F^{S,T}$ has a solution, then \eqref{eq:bkkdiscriminant} vanishes on $\bf L$ and $\bf d$.
  Recall from Lemma~\ref{lem:Pfaces} that the proper faces of $P^\ell$ are of the form $F_{S,T}= \operatorname{conv}({\bf e}_k-2{\bf e}_j : k\in S,\, j\in T)$ and that the proper faces of $\widetilde P^\ell$ are $\{{\bf 0}\}$, $P^\ell$, and $\widetilde F_{S,T} = {\rm conv}({\bf 0}, F_{S,T})$, for nonempty and disjoint $S,T\subset [\ell]$.

  First, if $a_k - 2a_j > 0$ for $k \in S$ and $j \in T$, then the corresponding subsystem is $\mathcal F^{S,T} = \{d_i=0\}_{i \in S \cup T}$. Thus, if $\mathcal F^{\bf a}$ has a root, then \eqref{eq:bkkdiscriminant} vanishes.   

  Next, if $a_k - 2a_j = 0$ for $k \in S$ and $j \in T$, then $F_{\bf a}(\widetilde P^\ell_k) = {\rm conv}({\bf 0}, F_{\bf a}(P^\ell_k)) = \widetilde F_{S,T} \cap \widetilde P^\ell_k$ for all $k \in S \cup T$ by Lemma~\ref{lem:Pifaces}.
  Hence, the subsystem $\mathcal F^{S,T}$ of $\mathcal F^{\bf a}$ is
  \begin{equation}\label{eq:FST}
    -4d_i f_{i,\bf a}^\ell({\bf x}) =
    \begin{cases}
      \sum_{k \in T} \frac{-L_{ikk}x_i}{x_k^2} + \sum_{j \neq k \in T} 2L_{ijk}
    \left (\frac{x_k}{x_ix_j} + \frac{x_j}{x_ix_k} - \frac{x_i}{x_jx_k} \right ) + 4d_i
    & \!\!\textrm{if $i \in S$}, \\
      \sum_{k \in S} \frac{2L_{iik}x_k}{x_i^2} + \sum_{j \neq k \in T} 2L_{ijk}
    \left (\frac{x_k}{x_ix_j} + \frac{x_j}{x_ix_k} - \frac{x_i}{x_jx_k} \right ) + 4d_i  
    & \!\! \textrm{if $i \in T$}. \\
    \end{cases}
  \end{equation}
  If \eqref{eq:FST} has a root, then 
    $\sum_{i \in T} -d_i f^\ell_{i,{\bf a}}({\bf x}) + 2  \, \sum_{j \in S} -d_j f^\ell_{j,{\bf a}}({\bf x}) =
    \sum_{i \in T} d_i + 2\sum_{j \in S} d_j$ vanishes, and hence so does \eqref{eq:bkkdiscriminant}.
  
  Finally, if $a_k - 2a_j < 0$ for $k \in S$ and $j \in T$, then $F_{\bf a}(\widetilde P^\ell_k) = F_{S,T} \cap P^\ell_k$ for all $k \in S \cup T$, by Lemma~\ref{lem:Pifaces}.
  In this case, the corresponding facial system $\mathcal F^{S,T}$ is
  \begin{equation}\label{eq:RST}
    -4 d_i r_{i,\bf a}^\ell({\bf x}) =
    \begin{cases}
      \sum_{k \in T} \frac{-L_{ikk}x_i}{x_k^2} + \sum_{j \neq k \in T} 2L_{ijk}
    \left (\frac{x_k}{x_ix_j} + \frac{x_j}{x_ix_k} - \frac{x_i}{x_jx_k} \right )
    & \textrm{if $i \in S$}, \\
      \sum_{k \in S} \frac{2L_{iik}x_k}{x_i^2} + \sum_{j \neq k \in T} 2L_{ijk}
    \left (\frac{x_k}{x_ix_j} + \frac{x_j}{x_ix_k} - \frac{x_i}{x_jx_k} \right )
    & \textrm{if $i \in T$},
    \end{cases}
  \end{equation}
  unless ${\bf a} = \bf 1$, in which case, $\mathcal F^{S,T} = \{r_{i}^\ell({\bf x}) = 0\}$. 
  We now remark that the restriction of $r^\ell_i$ to the face $F_{S,T} \cap  P_i^\ell = \emptyset$ is identically zero if $i \not\in S \cup T$. 
  Therefore, by \eqref{eq:ri-toricder}, the system \eqref{eq:RST} is precisely the system $-d_i r^\ell_{i,\bf a}({\bf x}) =  x_i\frac{\partial}{\partial x_i}\operatorname{scal}_{{\bf a}}({\bf x})=0$ for $i \in [\ell]$, where $\operatorname{scal}_{ {\bf a}}$ denotes the restriction of $\operatorname{scal}$ to the face $F_{S,T}$. 
  Thus, if \eqref{eq:RST} has a solution in $(\CC^*)^\ell$, then the coefficient vector $\bf L$ of $-4\operatorname{scal}$ lies in $\nabla_{A\cap F_{S,T}}$, and so $E_A(\operatorname{scal})$ and \eqref{eq:bkkdiscriminant}~vanish~on~$\bf L$~by~\eqref{eq:Adet-factors}.

  We have shown that if any facial system $\mathcal F^{\bf a}$ has a root in $(\CC^*)^\ell$, then \eqref{eq:bkkdiscriminant} vanishes on $\bf L$ and $\bf d$. The conclusion now follows from \Cref{thm:other}.
\end{proof}

The polytope $P^\ell$ is simple if and only if $\ell = 2,3$, by Lemma~\ref{lem:Pfaces}.
Therefore $P^2$ and $P^3$ define smooth toric varieties by \cite[Cor.~3.10]{PVL}.
For $\ell = 2$, the polynomial \eqref{eq:bkkdiscriminant} is equal to the discriminant in Proposition~\ref{prop:l=2bkk}. 
For $\ell = 3$, we compute the principal $A$-determinant $E_A(\operatorname{scal})$ applying \eqref{eq:Adet-factors} to the polygon $P^3$ in Figure~\ref{fig:newton}, obtaining
\begin{equation}\label{eq:l=3-gen-disc}
  E_A(\operatorname{scal} ) = 
  L_{112} L_{122} L_{113} L_{133} L_{223} L_{233}
  \begin{vmatrix}
    L_{123} & L_{133}\\
    L_{122} & L_{123}
  \end{vmatrix}
  \begin{vmatrix}
    L_{123} & L_{233}\\
    L_{112} & L_{123}
  \end{vmatrix}
  \begin{vmatrix}
    L_{123} & L_{113}\\
    L_{223} & L_{123}
  \end{vmatrix}
  \Delta_A,
\end{equation}
where $\Delta_A$ is the $A$-discriminant for the matrix $A\in \mathds Z^{3\times 12}$ evaluated at the vector ${\bf L}$ in Example~\ref{ex:Asmalll};
it can be computed explicitly using \cite[Thm.~2]{khetan} by evaluating the determinant of a $21 \times 21$ matrix.
The degree-one factors in \eqref{eq:l=3-gen-disc} are the $A$-discriminants of the vertices of $P^3$. The $2 \times 2$ determinants are the $A$-discriminants of edges $F_{S,T}$ of $P^3$ with $\#S = 2$ and $\#T = 1$, i.e., the short edges; see \Cref{ex:short-long}. The long edges, namely $F_{S,T}$ with $\#S = 1$ and $\#T = 2$, have $A$-discriminants with codimension $2$, so they each contribute a factor of $1$. Finally, the factor $\Delta_A$ is the $A$-discriminant of the $2$-dimensional face of $P^3$. 

\subsection{Finiteness}
Using Theorem~\ref{thm:bkkdiscriminant}, we confirm the Finiteness Conjecture in special cases.

\begin{proof}[Proof of Corollary~\ref{cor:fin}]
Let $\G/\mathsf H$ be a compact homogeneous space and $\mathfrak m=\mathfrak m_1\oplus\dots \mathfrak m_\ell$ a $Q$-orthogonal decomposition into pairwise inequivalent $\mathrm{Ad}_{\sf H}$-irreducible representations, with associated parameters $\bf b$, $\bf d$, and $L$, as in Section~\ref{sec:diffgeom}. 
If the principal $A$-determinant $E_A(\operatorname{scal} )$ does not vanish on $\bf b$, $\bf d$, $L$, then \eqref{eq:bkkdiscriminant} does not vanish, since $d_i>0$ for all $i\in[\ell]$. Thus, by Theorem~\ref{thm:bkkdiscriminant}, the system \eqref{eq:system} with these parameters is BKK generic, hence it has exactly $D_{\ell-1}$ solutions in $(\CC^*)^\ell$ by Theorems~\ref{thm:bkk} and \ref{thm:other}. Among those, the solutions that lie in $\RR^\ell_+$ are in bijective correspondence with the $\G$-invariant Einstein metrics on $\sf G/H$.
\end{proof}

As seen above, the BKK discriminant gives conditions under which a system has finitely many solutions.
However, being BKK generic is stronger than having finitely many solutions, cf.~Propositions~\ref{prop:l=2bkk} and \ref{prop:l=2finite}.
This distinction is relevant because there are compact homogeneous spaces with pairwise inequivalent irreducible summands whose homogeneous Einstein equations are \emph{not} BKK generic; see \Cref{sec:numerics} for examples.
Thus, attempting to establish BKK genericity is not a viable path to prove the Finiteness Conjecture.

Recall that, for $\ell = 2$, positivity of $\bf d$ is a sufficient condition for finiteness (Proposition~\ref{prop:l=2finite}).
We ask if this holds for $\ell\geq3$; an affirmative answer would imply the Finitness Conjecture.

\begin{question}
  Does \eqref{eq:system-nice} have finitely many solutions in $(\CC^*)^\ell$ if the entries of $\bf d$ are positive?
\end{question}

\subsection{Sharpness}
Theorem~\ref{thm:generic-finiteness} states that the BKK bound is achieved for generic parameters.
We now ask if this bound is achieved in practice. 
For $\ell = 2$, there are infinitely many examples of homogeneous spaces where the BKK bound is achieved. 

\begin{example}\label{ex:somn}
Recall that the outer tensor product of the defining representations of $\SO(m)$ and~$\SO(n)$ is the $\SO(m)\times \SO(n)$-representation on $\RR^m\otimes\RR^n\cong \RR^{mn}$ given by
\[ (A,B) \cdot (v\otimes w) = Av\otimes Bw \quad \text{ for all } (A,B)\in \SO(m)\times \SO(n), \; v\in\RR^m, w\in\RR^n. \]
This defines an injective homomorphism $\SO(m)\times \SO(n)\to \SO(mn)$ whose image is a maximal subgroup $\mathsf H$. Consider the homogeneous space $\G/\mathsf H=\SO(mn)/\SO(m)\SO(n)$ for $m,n\geq3$, $(m,n)\neq(4,4)$, cf.~\cite[Sec.~6, {\bf V}.1]{dk2}.
Fix the bi-invariant metric $Q(X,Y)=-\frac12\operatorname{tr}XY$ on $\mathfrak g=\mathfrak{so}(mn)$.
Using the standard identification $\mathfrak{so}(k)\cong\wedge^2\RR^k$, and the decomposition
\begin{equation*}
    \wedge^2(V\otimes W)=\wedge^2 V\oplus\wedge^2W\oplus (\operatorname{Sym}^2_0 V\otimes \wedge^2 W)\oplus (\wedge^2 V\otimes \operatorname{Sym}^2_0W),
\end{equation*}
we find the $Q$-orthogonal splitting $\mathfrak g=\mathfrak h\oplus\mathfrak m_1\oplus \mathfrak m_2$, where
\begin{equation*}
\mathfrak m_1 \cong \operatorname{Sym}^2_0\RR^m\otimes \wedge^2\RR^n, \qquad
\mathfrak m_2 \cong \wedge^2\RR^m\otimes \operatorname{Sym}^2_0\RR^n,
\end{equation*}
are inequivalent irreducible $\mathrm{Ad}_{\mathsf H}$-representations.
Thus, $\ell=2$ and the $d_i=\dim \mathfrak m_i$ are
\begin{equation*}
d_1 = \tfrac{(m + 2) (m - 1)}{2}\tfrac{n (n- 1)}{2},\qquad d_2 = \tfrac{m (m- 1)}{2} \tfrac{(n + 2) (n - 1)}{2}.
\end{equation*}
Using some representation theory, one computes ${\bf b}=(2 (m n-2)){\bf 1}$, as well as
\begin{equation}\label{eq:Ls-sonm}
\begin{aligned}
    L_{111}&= L_{222}=\tfrac{1}{8 m}\,(m-2) \,(m-1)\, (m+2)\, (m+4)\, n \,(n-2) \,(n-1),\\
    L_{112}&= L_{122}=\tfrac{1}{8} \,(m-1) \,m\, (m+2)\, (n-2)\, (n-1)\, (n+2).
\end{aligned}
\end{equation}
With the above values, one has $(2d_1 + d_2)\, (d_1 + 2d_2) \, R(r_1^2, r_2^2)>0$ for all $m,n\geq3$, so, by \Cref{prop:l=2bkk}, the system \eqref{eq:system-nice} has $D_{1}=3$ solutions in $(\CC^*)^2$, counted with multiplicity.
\end{example}

We believe that, for all $\ell\geq3$, there exist compact homogeneous spaces $\G/\mathsf H$ whose isotropy representation has a $Q$-orthogonal decomposition \eqref{eq:splitting-m} with $\ell$ pairwise inequivalent summands, such that the corresponding parameters lie outside the BKK discriminant.
On such homogeneous spaces, the BKK bound in Theorem~\ref{thm:delannoy} is achieved, and the Finiteness Conjecture holds.
However, we do not know how to produce explicit examples with $\ell\geq 3$, given the difficulty of computing structure constants $L_{ijk}$ in larger examples.

The mixed volume of \eqref{eq:system-nice} drops if some structure constant $L_{iik}$, $i\neq k\in [\ell]$, vanishes, in which case the number of isolated solutions to \eqref{eq:system-nice} is strictly less than $D_{\ell-1}$; see \Cref{sec:applications} for examples. In other words, BKK genericity requires that $L_{iik}$ be nonzero for all $i\neq k\in [\ell]$. Note that if $\mathsf H\subset \G$ is a maximal subgroup, as in \Cref{ex:somn}, then for all $i\in [\ell]$ there exists $k \in [\ell] \backslash \{i\}$ such that $L_{iik}>0$. Thus, if $\ell=2$, we have $L_{112}, L_{122}>0$, but, for $\ell\geq3$, maximality of $\mathsf H\subset \G$ no longer ensures that \emph{all} $L_{iik}>0$. This leads us to our last question:

\begin{question}
Construct examples of compact homogeneous spaces $\G/\mathsf H$ whose isotropy representation has a $Q$-orthogonal decomposition \eqref{eq:splitting-m} with $\ell\geq3$ pairwise inequivalent summands such that $L_{iik}>0$ for all $i\neq k\in [\ell]$. 
Check if the systems \eqref{eq:system-nice} corresponding to these examples achieve the BKK bound $D_{\ell-1}$ 
using Corollary~\ref{cor:fin}.
\end{question}

\section{The Finiteness Conjecture for generalized Wallach spaces}\label{sec:applications}
Generalized Wallach spaces are compact homogeneous spaces $\G/\mathsf H$ whose isotropy representation $\mathfrak m=\mathfrak m_1\oplus \mathfrak m_2\oplus \mathfrak m_3$ splits into $\ell=3$ pairwise orthogonal $\mathrm{Ad}_{\mathsf H}$-irreducible summands with $[\mathfrak m_i,\mathfrak m_i]\subset\mathfrak h$ for all $i \in [3]$.
Thus, the structure constants of generalized Wallach spaces satisfy $L_{iii} = L_{iik}=0$ for all $i, k \in [3]$.
These spaces are natural generalizations of the so-called Wallach flag manifolds
\[ \mathsf{SO}(3)/\ZZ_2\oplus\ZZ_2, \quad \mathsf{SU(3)}/\mathsf T^2, \quad \mathsf{Sp}(3)/\mathsf{Sp}(1)^3, \quad \mathsf F_4/\mathsf{Spin}(8), \]
that is, the manifolds of complete flags in $\RR^3$, $\CC^3$, $\HH^3$, and $\mathds{C}\mathrm{a}^3$. Generalized Wallach spaces were classified by Nikonorov~\cite{niko-class1,niko-class2}, and independently in \cite{class3} for $\G$ simple. 

The homogeneous Einstein equations simplify on generalized Wallach spaces.
Consider the system \eqref{eq:system-nice} with $\ell = 3$, as well as $L_{iik} = 0$ and $L'_{iii} =  - 2b_id_i$, for all $i \neq k \in [3]$, that is:
\begin{equation}\label{eq:l=3nice}
\begin{aligned}
    \frac{L_{111}'}{x_1} + 2L_{123} \left(\frac{x_2}{x_1x_3} + \frac{x_3}{x_1x_2} - \frac{x_1}{x_2x_3}\right) + 4d_1 = 0,\\ 
    \frac{L_{222}'}{x_2} + 2L_{123} \left(\frac{x_1}{x_2x_3} + \frac{x_3}{x_1x_2} - \frac{x_2}{x_1x_3}\right) + 4d_2 = 0,\\ 
    \frac{L_{333}'}{x_3} + 2L_{123} \left(\frac{x_1}{x_2x_3} + \frac{x_2}{x_1x_3} - \frac{x_3}{x_1x_2}\right) + 4d_3 = 0.
\end{aligned}  
\end{equation}
The system \eqref{eq:l=3nice} is supported on the simplex
\begin{equation}\label{eq:l=3supp}
\conv({\bf 0}, \, {\bf e}_1 - {\bf e}_2 - {\bf e}_3,\, {\bf e}_2 - {\bf e}_1 - {\bf e}_3,\, {\bf e}_3 - {\bf e}_1 - {\bf e}_2),
\end{equation}
whose normalized volume is 4. 
Therefore \eqref{eq:l=3nice} has BKK bound $4$, instead of $D_2=13$ for the largest support of \eqref{eq:system-nice} with $\ell=3$. Homogeneous Einstein metrics on generalized Wallach spaces have been studied in many papers, e.g., \cite{kimura,arva-tams,gen-wallach-einstein,class3,CN} among others. In particular, according to \cite[Thm.~1]{gen-wallach-einstein}, there are at most 4 solutions on each such space. Applying Theorems~\ref{thm:bkk} and \ref{thm:other} to \eqref{eq:l=3nice}, we provide an alternative proof:

\begin{theorem}\label{thm:CN-finiteness}
Each generalized Wallach space $\G/\mathsf H$ in the classification of Nikonorov~\cite{niko-class1,niko-class2} carries at most 4 distinct (diagonal) $\G$-invariant Einstein metrics.
\end{theorem}

In particular, the Finiteness Conjecture holds for generalized Wallach spaces.
The main tool we use to prove \Cref{thm:CN-finiteness} is the following BKK discriminant:

\begin{proposition}\label{prop:l=3}
If $d_1, d_2, d_3 > 0$, and
  \begin{equation}\label{eq:l=3discriminant}
    L_{123}
    \begin{vmatrix}
      4L_{123} & L_{111}'\\
      L_{111}' & 4L_{123}
    \end{vmatrix}
    \begin{vmatrix}
      4L_{123} & L_{222}'\\
      L_{222}' & 4L_{123}
    \end{vmatrix}
    \begin{vmatrix}
      4L_{123} & L_{333}'\\
      L_{333}' & 4L_{123}
    \end{vmatrix}
    \begin{vmatrix}
      4L_{123} & L_{111}' & L_{222}'\\
      L_{111}' & 4L_{123} & L_{333}'\\
      L_{222}' & L_{333}' & 4L_{123}
    \end{vmatrix}
  \end{equation}
  does not vanish, then \eqref{eq:l=3nice} has exactly $4$ solutions in $(\CC^*)^3$, counted with multiplicity.
\end{proposition}

\begin{proof}
  For each face of the simplex \eqref{eq:l=3supp}, we compute its facial system and corresponding resultant in the parameters $L_{111}'$, $L_{222}'$, $L_{333}'$, $L_{123}$, $d_1$, $d_2$, and $d_3$.
  The facial systems of the vertices all have resultant $L_{123}$.
  Each face $\conv({\bf e}_i - {\bf e}_j - {\bf e}_k,\, {\bf e}_k - {\bf e}_i - {\bf e}_j)$ contributes the resultant $(L_{jjj}')^2 - 16L_{123}^2$.
  The resultant of the face $\conv({\bf e}_1 - {\bf e}_2 - {\bf e}_3, \, {\bf e}_2 - {\bf e}_1 - {\bf e}_3, \, {\bf e}_3 - {\bf e}_1 - {\bf e}_2)$ is the $3\times 3$ determinant in \eqref{eq:l=3discriminant}.
  The ideals of the faces $\conv({\bf 0},  {\bf e}_i - {\bf e}_j - {\bf e}_k)$ and $\conv({\bf 0}, \, {\bf e}_i - {\bf e}_j - {\bf e}_k, \, {\bf e}_j - {\bf e}_i - {\bf e}_k)$ contain the relation $d_i + d_j$, which is never zero since all $d_i>0$.
  As \eqref{eq:l=3supp} has normalized volume 4, the conclusion follows by Theorems~\ref{thm:bkk}~and~\ref{thm:other}.
\end{proof}

We observe that, by \eqref{eq:Adet-factors}, the polynomial \eqref{eq:l=3discriminant} is the principal $A$-determinant of
\begin{equation*}
 - 4 \operatorname{scal}({\bf x}) = \frac{L_{111}'}{x_1} +
  \frac{L_{222}'}{x_2} +
  \frac{L_{333}'}{x_3} +
  2L_{123}\left(\frac{x_1}{x_2x_3} + \frac{x_2}{x_1x_3} + \frac{x_3}{x_1x_2}\right).
\end{equation*}

\begin{proof}[Proof of \Cref{thm:CN-finiteness}]
According to Nikonorov~\cite{niko-class1,niko-class2}, there are four types of generalized Wallach spaces $\G/\mathsf H$; we analyze them below using the same labels. In all cases below, we refer to the system \eqref{eq:l=3nice} with $b_i=1$, hence $L'_{iii}=-2d_i$, for all $i\in [3]$.

\smallskip
\noindent
Type (1). These spaces are products of three irreducible symmetric spaces of compact type, so all $L_{ijk}=0$, and \eqref{eq:l=3nice} has a unique solution ${\bf x}=\big(\frac12,\frac12,\frac12\big)$ in $(\CC^*)^3$.

\renewcommand{\arraystretch}{1.5}
\begin{table}[htp]
\begin{center}
\resizebox{\textwidth}{!}{
\begin{tabular}{|c||c|c|c|c|c|c|}
\hline
& \ $\mathfrak{g}$& $\mathfrak{h}$ & $d_1$ &  $d_2$ & $d_3$  & $L_{123}$   \\
\hline\hline
1&  $\mathfrak{so}(k + l + m)$   &  $\mathfrak{so}(k) \oplus  \mathfrak{so}(l)  \oplus  \mathfrak{so}(m)$  & $kl$ & $km$ & $lm$ &
  $ \frac{klm}{2(k + l + m - 2)} $ \\  \hline
2&  $\mathfrak{su}(k + l + m)$   &$  \mathfrak{s(u}(k) \oplus  \mathfrak u(l)  \oplus \mathfrak u(m))$  & $2kl$ & $2km$ & $2lm$ &
  $  \frac{klm}{k + l + m} $ \\  \hline
3&  $\mathfrak{sp}(k + l + m)$   &  $\mathfrak{sp}(k) \oplus  \mathfrak{sp}(l)  \oplus  \mathfrak{sp}(m)$  & $4kl$ & $4km$ & $4lm$ &  $\frac{2klm}{k + l + m+1}$ \\  \hline
4&$\mathfrak{su}(2l)$, \,$l\geq 2$ &$\mathfrak u(l)$&  $l(l-1 )$  &  $l(l+1)$  &  $l^2-1$  &  $\frac{l(l^2-1)}{4}$ \\  \hline
5&$\mathfrak{so}(2l)$, \,$l\geq 4$ &$\mathfrak u(1)\oplus \mathfrak u(l-1)$ &  $2(l-1 )$   &  $2(l-1 )$  &  $(l - 1) (l - 2)$  &  $\frac{l-1}{2}$ \\  \hline \hline
6&$\mathfrak e_6$ &$\mathfrak{su}(4)\oplus \mathfrak{sp}(1)^2\oplus \RR$   &$16$ & $16$ & $24$ &  $4$ \\  \hline
7&$\mathfrak e_6$ &$\mathfrak{so}(8)\oplus \RR^2$ & $16$ & $16$ & $16$ &  $\frac{8}{3}$ \\  \hline
8&$\mathfrak e_6$ &$\mathfrak{sp}(3)\oplus \mathfrak{sp}(1)$ & $14$ & $28$ & $12$ &  $\frac{7}{2}$ \\  \hline
9&$\mathfrak e_7$ &$\mathfrak{so}(8)\oplus \mathfrak{sp}(1)^3$ & $32$ & $32$ & $32$ &  $\frac{64}{9}$ \\  \hline
10&$\mathfrak e_7$ &$\mathfrak{su}(6)\oplus \mathfrak{sp}(1)\oplus \RR$ & $30$ & $40$ & $24$ &  $\frac{20}{3}$ \\  \hline
11&$\mathfrak e_7$ &$\mathfrak{so}(8)$ & $35$ & $35$ & $35$ &  $\frac{175}{18}$ \\  \hline
12&$\mathfrak e_8$ &$\mathfrak{so}(12)\oplus \mathfrak{sp}(1)^2$ & $64$ & $64$ & $48$ &  $\frac{64}{5}$ \\  \hline
13&$\mathfrak e_8$ &$\mathfrak{so}(8)\oplus \mathfrak{so}(8)$ & $64$ & $64$ & $64$ &  $\frac{256}{15}$ \\  \hline
14&$\mathfrak f_4$ &$\mathfrak{so}(5)\oplus \mathfrak{sp}(1)^2$ & $8$ & $8$ & $20$ &  $\frac{20}{9}$ \\  \hline
15&$\mathfrak f_4$ &$\mathfrak{so}(8)$ & $8$ & $8$ & $8$ &  $\frac{8}{9}$ \\  \hline
\end{tabular}
}
\end{center}
\caption{Generalized Wallach spaces $\G/\mathsf H$ with $\G$ simple, from \cite[Table 1]{niko-class1}.}\label{tab:genwallach1}
\end{table}
\renewcommand{\arraystretch}{1}%

\smallskip
\noindent
Type (2). These spaces have $\G$ simple and are listed in \cite[Tab.~1]{niko-class1}, in terms of the Lie algebras $\mathfrak g$ and $\mathfrak h$, which we reproduce in \Cref{tab:genwallach1} using our notation.
The first 5 rows are infinite families involving classical Lie algebras, and we may assume $k\geq l\geq m\geq 1$ in rows 1-3. Using cylindrical algebraic decomposition on a computer algebra system, e.g., the command \texttt{Reduce} in \texttt{Mathematica}, one verifies that the only cases where \eqref{eq:l=3discriminant} vanishes are row 1 if $l=m=1$, and row 4 with $l=2$ or $l=3$. In all these cases, the system \eqref{eq:l=3nice} can be solved explicitly and there are $3$ or $4$ solutions in $(\CC^*)^3$. The remaining 10 rows are sporadic examples involving exceptional Lie algebras, and \eqref{eq:l=3discriminant} does not vanish in all such cases.

\smallskip
\noindent
Type (3). These are so-called Ledger--Obata spaces $\G/\mathsf H$ with $\G=\mathsf F\times \mathsf F\times \mathsf F\times \mathsf F$ and $\mathsf H=\Delta \mathsf F$, where $\mathsf F$ is a simple Lie group. In this situation, $d_i=\dim \mathsf F$ and $L_{123}=\tfrac14\dim \mathsf F$, so one easily checks that \eqref{eq:l=3discriminant} does not vanish. In fact, \eqref{eq:l=3nice} has a unique solution ${\bf x}=\big(\frac38,\frac38,\frac38\big)$ in $(\CC^*)^3$.

\smallskip
\noindent
Type (4). These spaces are $\G/\mathsf H$ with $\G=\mathsf F\times \mathsf F$ and $\mathsf H=\Delta \mathsf K\subset \mathsf K\times \mathsf K$, where $(\mathsf F,\mathsf K)$ is an irreducible symmetric pair of compact type, with $\mathsf F$ simple and $\mathsf K$ simple or 1-dimensional. In this case, $d_1=d_2=\dim \mathsf F/\mathsf K$, $d_3=\dim \mathsf K$, and $L_{123}=\tfrac14 \dim \mathsf F/\mathsf K$, so \eqref{eq:l=3discriminant} vanishes if and only if $\dim \mathsf F = 3 \dim \mathsf K$. The only symmetric pairs $(\mathsf F, \mathsf K)$ as above with $\dim \mathsf F = 3 \dim \mathsf K$ are $(\mathsf{SU}(2), \mathsf{SO}(2))$, $(\mathsf{Sp}(1), \mathsf{SO}(2))$, and $(\mathsf{SO}(3), \mathsf{SO}(2))$; in all these cases, the system \eqref{eq:l=3nice} can be solved explicitly and there are 3 solutions in $(\CC^*)^3$.
\end{proof}

\begin{remark} 
According to \cite{CN}, homogeneous Einstein metrics on generalized Wallach spaces of type (2) had been classified except for row one, i.e., $\mathsf{SO}(k+l+m)/\mathsf{SO}(k) \mathsf{SO}(l) \mathsf{SO}(m)$, $k\geq l\geq m\geq 1$. 
Homogeneous Einstein metrics on spaces of type~(3), including \emph{nondiagonal} ones, were classified in \cite{cnn2}. 
The existence of type~(4) was only noticed~years~later~\cite{niko-class2}. 
\end{remark}

\section{Numerical experiments on full flag manifolds}\label{sec:numerics}

In this section, we count and compute 
$\G$-invariant Einstein metrics on the full flag manifolds $\G/\mathsf H$, where $\G$ is a compact simple Lie group of classical type and $\mathsf H\subset \G$ is a maximal torus. We use the numerical algebraic geometry software \verb+HomotopyContinuation.jl+ \cite{hc}. 
These systems were previously studied using Gr\"obner basis techniques; see, e.g., \cite{guzman2024, sakane}.

\subsection{Numerical Algebraic Geometry}\label{subsec:nag}
We include a brief discussion of our numerical techniques.
We first use a monodromy method to solve \eqref{eq:system-nice} with generic parameters. 
The number of solutions to the generic system is equal to the BKK bound.
We then use a parameter homotopy to track the generic parameters to our special parameters while simultaneously tracking the solutions. 
The solutions to the generic and special systems are both certified using interval arithmetic \cite{BRT}; 
this produces a proof that there exists an actual solution within a certain radius of every floating-point solution, and that these solutions are distinct.
Therefore, this procedure yields a rigorous lower bound on the number of solutions to a system.

If the special system is BKK generic, then we have an upper bound as well, and hence a proof that we found all of the solutions.
Because we are using numerical methods, it is possible that some of the paths fail when tracking solutions from generic to special parameters. 
In this case, the number of certified solutions to the special system is smaller than the BKK bound and so we cannot prove that we have computed all solutions.
However, it is rare that a path to a true solution fails, so we have high confidence that the numbers in Table~\ref{tab:numerics} are the actual solution counts. 

We remark that numerical methods can handle much larger systems than symbolic methods. For instance, we could not produce Table~\ref{tab:numerics} using only Gr\"{o}bner bases techniques.

\subsection{Setup}
Let $\G$ be a compact simple Lie group of classical type of rank $n$, that is, one of the groups in Table~\ref{tab:rootsystems}.
Set $\mathsf H$ to be the standard maximal torus $\mathsf T^n\subset\G$ that determines the root system $\Phi=\Phi^+\cup(-\Phi^+)$, where $\Phi^+$ is the choice of positive roots listed in Table~\ref{tab:rootsystems}.

\begin{table}[ht]
\renewcommand{\arraystretch}{1.5}
\begin{center}
\resizebox{\textwidth}{!}{
  \begin{tabular}{|c||c|c|c|c|c|}
  \hline
  & $\G$ & $\Phi^+$  & $\ell=|\Phi^+|$ & Weyl group & Nonvanishing $L_{\alpha,\,\beta,\,\gamma}$\\
  \hline\hline
$\begin{array}{c}
  {\rm A}_n \\[-2ex] _{n\geq1}
\end{array}$ 
& $\mathsf{SU}(n+1)$  & $ \begin{array}{c} \varepsilon_i-\varepsilon_j \\[-2ex] _{i < j \in [n+1]} \end{array}$ & $\binom{n+1}{2}$ & $\mathsf S_{n+1}$  & $L_{\varepsilon_i-\varepsilon_k,\, \varepsilon_k-\varepsilon_j,\,\varepsilon_i-\varepsilon_j}=\frac{1}{n+1}$ \\ \hline
$\begin{array}{c}
  {\rm B}_n \\[-2ex] _{n\geq2}
\end{array}$ 
  & $\mathsf{SO}(2n+1)$ & $\begin{array}{c} \varepsilon_i \pm \varepsilon_j, \, \varepsilon_k \\[-2ex] _{i<j \in [n],\,\, k \in [n]} \end{array}$ & $n^2$ & $(\ZZ_2)^{n} \ltimes \mathsf S_n$  &
  $\begin{array}{r}
L_{\varepsilon_i - \varepsilon_j,\,\varepsilon_j - \varepsilon_k,\,\varepsilon_i-\varepsilon_k}=\frac{1}{2n-1} \\ 
L_{\varepsilon_i + \varepsilon_j,\,\varepsilon_j + \varepsilon_k,\,\varepsilon_i-\varepsilon_k}=\frac{1}{2n-1} \\ 
L_{\varepsilon_i - \varepsilon_j,\,\varepsilon_i,\,\varepsilon_j}=\frac{1}{2n-1}\\
L_{\varepsilon_i + \varepsilon_j,\,\varepsilon_i,\,\varepsilon_j}=\frac{1}{2n-1}
  \end{array}$ \\ \hline
$\begin{array}{c}
  {\rm C}_n \\[-2ex] _{n\geq3}
\end{array}$ 
& $\mathsf{Sp}(n)$  & $\begin{array}{c} \varepsilon_i \pm \varepsilon_j, \, 2\varepsilon_k \\[-2ex] _{i < j\in [n],\,\,  k\in [n]} \end{array}$ & $n^2$ & $(\ZZ_2)^{n} \ltimes \mathsf S_n$ & 
$\begin{array}{r}
L_{\varepsilon_i - \varepsilon_j,\,\varepsilon_j - \varepsilon_k,\,\varepsilon_i-\varepsilon_k}=\frac{1}{2n+2} \\ 
L_{\varepsilon_i - \varepsilon_j,\,\varepsilon_j + \varepsilon_k,\,\varepsilon_i + \varepsilon_k}=\frac{1}{2n+2} \\ 
L_{\varepsilon_i - \varepsilon_j,\,2\varepsilon_j,\,\varepsilon_i+\varepsilon_j}=\frac{1}{n+1}
\end{array}$ 
\\ \hline
$\begin{array}{c}
  {\rm D}_n \\[-2ex] _{n\geq4}
\end{array}$ 
& $\mathsf{SO}(2n)$  & $ \begin{array}{c} \varepsilon_i \pm \varepsilon_j \\[-2ex] _{i<j\in[n]} \end{array}$ & $2\binom{n}{2}$ & $(\ZZ_2)^{n-1} \ltimes \mathsf S_n$
&
$\begin{array}{r}
L_{\varepsilon_i-\varepsilon_j,\, \varepsilon_j - \varepsilon_k,\,\varepsilon_i - \varepsilon_k}=\frac{1}{2n-2}\\
L_{\varepsilon_i-\varepsilon_j,\, \varepsilon_j + \varepsilon_k,\,\varepsilon_i + \varepsilon_k}=\frac{1}{2n-2}\\
\end{array}$   \\   \hline
  \end{tabular}
  }
  \end{center}

  \medskip
\caption{Compact simple Lie groups of classical type and positive roots.
The last column lists the nonvanishing structure constants $L_{\alpha,\,\beta,\,\gamma}$ for $\alpha,\beta,\gamma$, up to permuting $\{\alpha,\beta,\gamma\}$ and all sign changes. For details, see \cite{sakane}.
}\label{tab:rootsystems}
\end{table}

Fix the bi-invariant metric $Q=-B$ given by the negative of the Cartan--Killing form on $\mathfrak g$, and denote by $\mathfrak h_\CC \subset \mathfrak g_{\CC}$ the complexifications of $\mathfrak h\subset\mathfrak g$.
Given a linear functional $\alpha \colon \mathfrak h_{\CC}\to\RR$, set $\mathfrak g_{\alpha} \coloneqq \{ X\in\mathfrak g_{\CC} : [H,X]= \sqrt{-1}\,\alpha(H) X \text{ for all } H \in \mathfrak h_{\CC} \}$, and recall the decomposition
\[\mathfrak g_{\CC}=\mathfrak h_{\CC}\oplus \bigoplus_{\alpha\in\Phi} \mathfrak g_\alpha,\]
where $\dim_{\CC}\mathfrak g_{\alpha}=1$ for all roots $\alpha \in \Phi$; see, e.g.,~\cite[\S 4.3]{mybook}. Moreover, $\overline{\mathfrak g_{\alpha}}=\mathfrak g_{-\alpha}$, and $[\mathfrak g_{\alpha}, \mathfrak g_{\beta}]\subset \mathfrak g_{\alpha +\beta}$ for all $\alpha,\beta \in \Phi$, where $\mathfrak g_0 = \mathfrak h_{\CC}$, and $B(\mathfrak g_{\alpha}, \mathfrak g_{\beta})=0$ if $\alpha + \beta\neq0$.
Thus, the $Q$-orthogonal complement $\mathfrak m$ to the subalgebra $\mathfrak h\subset\mathfrak g$ decomposes as the direct sum 
\begin{equation*}
  \mathfrak m = \bigoplus_{\alpha\in\Phi^+} \mathfrak m_\alpha,
\end{equation*}
where $\mathfrak m_\alpha \coloneqq (\mathfrak g_{\alpha}\oplus\mathfrak g_{-\alpha})\cap \mathfrak g$, for all $\alpha\in\Phi$, are irreducible $\mathrm{Ad}_{\mathsf H}$-invariant representations. 
Note that $\mathfrak m_\alpha=\mathfrak m_{-\alpha}$, and $\mathfrak m_\alpha\not\cong\mathfrak m_\beta$ if $\alpha\neq\pm \beta$ since $\mathrm{Ad}(\exp X)$ is a rotation by angle $\alpha(X)$ on $\mathfrak m_{\alpha}\cong \RR^2$, for all $X\in \mathfrak h$ and $\alpha\in\Phi^+$. Thus, in this section, we write \eqref{eq:splitting-m} replacing indices $i\in [\ell]$ with indices $\alpha\in\Phi^+$; accordingly, we write ${\bf x}=(x_i)_{i\in [\ell]}$ as ${\bf x}=(x_\alpha)_{\alpha\in \Phi^+}$. Moreover, $d_\alpha=\dim_{\RR}\mathfrak m_\alpha=2$ and $b_\alpha=1$ for all $\alpha\in\Phi^+$, i.e., ${\bf d}={\bf 2}$ and $\bf b=\bf 1$. The nonvanishing structure constants $L_{\alpha,\,\beta,\,\gamma}$ are given in Table~\ref{tab:rootsystems}. Note that $L_{\alpha,\,\beta,\,\gamma}=0$ unless $\gamma=\alpha\pm \beta$, up to permuting $\{\alpha,\beta,\gamma\}$ and changing signs, since $[\mathfrak m_{\alpha},\mathfrak m_{\beta}]\subset \mathfrak m_{\alpha \pm \beta}$. 

Each choice of positive roots, or, equivalently, choice of a Weyl chamber, corresponds to a choice of $\G$-invariant complex structure on $\sf G/H$. For each such choice, there is a unique $\G$-invariant K\"ahler-Einstein metric on $\sf G/H$ compatible with that complex structure; see \cite[Thm.~8.95]{besse} or \cite[Lem.~3]{sakane}. As a solution ${\bf x}$ to \eqref{eq:system} with the choice $\Phi^+$ of positive roots, the K\"ahler-Einstein metric is characterized by the property that $x_{\alpha+\beta}= x_{\alpha}+x_{\beta}$ for all \emph{primitive} roots $\alpha,\beta\in\Phi^+$, i.e., roots $\alpha,\beta\in\Phi^+$ that are not sums of other positive roots.
The gauge group $\mathsf{N(H)/H}$ is the Weyl group listed in Table~\ref{tab:rootsystems}, whose action permutes $\Phi$ and hence the variables ${\bf x}=(x_\alpha)$, producing isometric metrics; see Section~\ref{subsec:gaugegroup}. This $\mathsf{N(H)/H}$-action on $\bf x$ leaves the system \eqref{eq:system} invariant, as it simply permutes its equations. Note that the $\mathsf{N(H)/H}$-orbit of the K\"ahler-Einstein metric consists of metrics that are K\"ahler with respect to the $\G$-invariant complex structures corresponding to other Weyl chambers.

The normal homogeneous metric $Q|_{\mathfrak m}$ on $\G/\mathsf H$ is Einstein, for some Einstein constant $\lambda>0$, if and only if all roots $\alpha \in \Phi$ have the same length~\cite[Cor~1.5]{wang-ziller-ens}. By Table~\ref{tab:rootsystems}, this occurs only in types ${\rm A}_n$ and ${\rm D}_n$. In these cases, ${\bf x}= \lambda \, {\bf 1}$ is a solution to \eqref{eq:system}.

\subsection{Counting and computing Einstein metrics}
Our numerical results on full flag manifolds $\sf G/H$ are summarized in Theorem~\ref{thm:examples} and Table~\ref{tab:numerics}.

\begin{table}[ht]
\renewcommand{\arraystretch}{1.5}
\begin{center}
  \begin{tabular}{|c||c|c|c|c||c|c||c||c|}
    \hline
    Type of $\G$ & ${\rm A}_2$ & ${\rm A}_3$ & ${\rm A}_4$ & ${\rm A}_5$ & ${\rm B}_2$ & ${\rm B}_3$ & ${\rm C}_3$ & ${\rm D}_4$ \\
    \hline\hline
    BKK Bound      & $4$ & $80$ & $9\,168$ & $6\,603\,008$ & $12$ & $5376$ & $5232$ & $239\,744$\\
    \hline
    $\#$ solutions in $(\CC^*)^\ell$  & $4$ & $59$ & $7\,908$ &  $5\,037\,448$ & $10$ & $4224$ & $4512$ & $ 150\,256$\\
    \hline
    $\#$ solutions in $(\RR^*)^\ell$  & $4$ & $29$ & $1\,596$ & $191\,252$ & $6$ & $750$ & $728$ & $ 11\,128$\\    
    \hline
    \renewcommand{\arraystretch}{1}
    $\begin{array}{c}
    \# \, \text{solutions in }\RR_+^\ell\text{, i.e.,}\\
      \#\, \G\text{-invariant Einstein}\\\text{metrics on }\sf G/H
        \end{array}$
        \renewcommand{\arraystretch}{1.5}
     & $4$ & $29$ & $396$ & $6572$ & $6$ & $48$ & $64$ & $ 184$\\
    \hline
    \renewcommand{\arraystretch}{1}
    $\begin{array}{c}
        \#\, \text{isometry classes of}\\ \G\text{-invariant Einstein}\\\text{metrics on }\sf G/H
        \end{array}$
        \renewcommand{\arraystretch}{1.5}
     & 2 & 4 & 12 &  35 & 2 & 5 & 4 & 5\\
    \hline
  \end{tabular}
  \end{center}

  \medskip
  \caption{Isolated solutions (without multiplicity) to the Einstein equations \eqref{eq:system} on the full flag manifold $\sf G/H$, where $\G$ is a compact simple Lie group and $\sf H\subset G$ is a maximal torus, compared with BKK bound.}\label{tab:numerics}
\end{table}

Because the structure constants mostly take the same values (see Table~\ref{tab:rootsystems}), we do not expect the system to be BKK generic and this is confirmed by the difference between the first and second rows of Table~\ref{tab:numerics}.
Note that the BKK discriminants for these systems are different from the one in Theorem~\ref{thm:bkkdiscriminant}, since many structure constants vanish and hence the supports of these systems are smaller.
The last row in Table~\ref{tab:numerics} is obtained using the volume \eqref{eq:vol} to distinguish nonisometric solutions, and the action of the gauge group $\mathsf{N(H)/H}$ to recognize isometric solutions. For type ${\rm D}_4$, besides the gauge group $\mathsf{N(H)/H}\cong (\ZZ_2)^{3} \ltimes \mathsf S_4$, we also use the group of triality (outer) automorphisms to recognize isometric solutions. For $\G$ of type ${\rm A}_4$, ${\rm B}_3$, ${\rm C}_3$, and ${\rm D}_4$, we list the coefficients of a representative from each isometry class in Tables~\ref{tab:A4}, \ref{tab:B3}, \ref{tab:C3} and \ref{tab:D4}, respectively.

Consider the family ${\rm A}_n$, that is, the homogeneous space $\mathsf{SU(n+1)}/\mathsf T^n$. The case $n=1$ is simply the 2-sphere $S^2$, which has a unique Einstein metric. For $n=2$, this is the Wallach flag manifold and there are exactly 4 solutions, but only 2 up to isometries. Arvanitoyeorgos~\cite{arva-tams} showed that, if $n\geq 3$, there are at least $\tfrac{(n+1)!}{2}+n+2$ solutions; however, up to isometries, these solutions yield only 3 distinct homogeneous Einstein metrics \cite[\S 6]{arva-tams}.
The entries in column ${\rm A}_3$ may be found in \cite[p.~305]{graev2}. 
Recently, Guzman~\cite{guzman2024} announced that, for $n=4$, there are at least 7 nonisometric homogeneous Einstein metrics. 
The entries $396$ and $12$ in ${\rm A}_4$ and $35$ in ${\rm A}_5$ were computed using a different numerical method in \cite{GM16}.
The number $6572$ in the ${\rm A}_5$ column is an improvement on the number $3941$ in \cite[Thm.~4]{GM16}.
Because we compute all solutions for a general system and specialize to specific parameters, we can say with a high degree of confidence that we have found \emph{all} solutions. 
For type ${\rm B}_2$,
it is proven in \cite[p. 81]{sakane} that $\sf{SO(5)/T^2}$  admits precisely $2$ homogeneous Einstein metrics up to isometries.
The numbers $48$ and $5$ in column ${\rm B}_3$ were found in \cite{WLZ18} and the numbers $64$ and $4$ in column ${\rm C}_3$ were found in \cite{GW}.

\begin{proof}[Proof and discussion of Theorem~\ref{thm:examples}]
Let $\sf G/H$ be a full flag manifold where $\G$ is a compact simple Lie group of type ${\rm A}_n$, ${\rm B}_n$, ${\rm C}_n$, or ${\rm D}_n$. We use the homogeneous Einstein equations \eqref{eq:system} for $\sf G/H$ as written in \cite[p. 76, 80, 85, 86]{sakane}, respectively; see Remark~\ref{rem:typo}.
The first row was computed numerically with the \verb+Julia+ package \verb+MixedSubdivisions.jl+ \cite{ms}. 
The second, third, and fourth rows were computed numerically using \verb+HomotopyContinuation.jl+ \cite{hc}.
The numerical computations in rows 2 and 3 were certified using the \verb+certify+ command in \verb+HomotopyContinuation.jl+ \cite{BRT}.
For the last row, we use the volume of the floating-point solutions to distinguish isometry classes, and the action of the gauge group (and triality automorphisms in the case ${\rm D}_4$) to detect isometric solutions.
\end{proof}

\begin{remark}\label{rem:typo}
  The equations for type ${\rm B}_n$ given in \cite[p. 80]{sakane} contain a typo. Following \cite{sakane}, instead of $x_{\varepsilon_i - \varepsilon_j}$, $x_{\varepsilon_i + \varepsilon_j}$, and $x_{\varepsilon_k}$, label the variables corresponding to the positive roots $\varepsilon_i - \varepsilon_j$, $\varepsilon_i + \varepsilon_j$, $\varepsilon_k$ as $x_{ij} = x_{ji}$, $y_{ij} = y_{ji}$, $z_k$, respectively; that is, consider the metric
  \[-\sum_{i<j}x_{ij}\, B|_{{\mathfrak m}_{\varepsilon_i - \varepsilon_j}}-\sum_{i<j} y_{ij}\, B|_{{\mathfrak m}_{\varepsilon_i + \varepsilon_j}}-\sum_{k} z_{k}\, B|_{{\mathfrak m}_{\varepsilon_k}}.\] 
  The components $r_{\varepsilon_i - \varepsilon_j}$ and $r_{\varepsilon_i}$ of the Ricci tensor on $\mathfrak m_{\varepsilon_i - \varepsilon_j}$ and $\mathfrak m_{\varepsilon_i}$ are correct as stated in \cite[p. 80]{sakane}, but 
  the expression for the component $r_{\varepsilon_i + \varepsilon_j}$ on $\mathfrak m_{\varepsilon_i + \varepsilon_j}$ must be replaced with
  \begin{align*}
    r_{\varepsilon_j + \varepsilon_j} &=
    \frac{1}{2y_{ij}} + \frac{1}{4(2n - 1)}
    \left (
    \sum_{k \neq i,j} \left ( \frac{y_{ij}}{x_{ik}y_{jk}} - \frac{x_{ik}}{y_{ij}y_{jk}} - \frac{y_{jk}}{y_{ij}x_{ik}} \right ) \right .\\
    &\, +
    \left . \sum_{k \neq i,j} \left ( \frac{y_{ij}}{x_{jk}y_{ik}} - \frac{x_{jk}}{y_{ij}y_{ik}} - \frac{y_{ik}}{y_{ij}x_{jk}} \right )
    +
    \left ( \frac{y_{ij}}{z_i z_j} - \frac{z_i}{y_{ij}z_j} - \frac{z_j}{y_{ij}z_i} \right )
    \right ).
  \end{align*}
\end{remark}

\vfill

\pagebreak[4]
  \global\pdfpageattr\expandafter{\the\pdfpageattr/Rotate 90}
  
\begin{landscape}
\begin{table}
  \[
  \begin{array}{|c||c|c|c|c|c|c|c|c|c|c|c|c|c|c|}
    \hline
    & x_{\varepsilon_1 - \varepsilon_2} & x_{\varepsilon_1 - \varepsilon_3} &
    x_{\varepsilon_1 - \varepsilon_4} & x_{\varepsilon_1 - \varepsilon_5} & x_{\varepsilon_2 - \varepsilon_3} & x_{\varepsilon_2 - \varepsilon_4} & x_{\varepsilon_2 - \varepsilon_5} & x_{\varepsilon_3 - \varepsilon_4} & x_{\varepsilon_3 - \varepsilon_5} & x_{\varepsilon_4 - \varepsilon_5}\\
\hline \hline
1 & 0.35 & 0.35 & 0.35 & 0.35 & 0.35 & 0.35 & 0.35 & 0.35 & 0.35 & 0.35 \\
\hline
2 & 0.2 & 0.4 & 0.6 & 0.8 & 0.2 & 0.4 & 0.6 & 0.2 & 0.4 & 0.2 \\
\hline
3 & 0.4125 & 0.4125 & 0.4125 & 0.275 & 0.4125 & 0.4125 & 0.275 & 0.4125 & 0.275 & 0.275 \\
\hline
4 & 0.29293 & 0.29293 & 0.29293 & 0.29293 & 0.464 & 0.28514 & 0.464 & 0.464 & 0.28514 & 0.464 \\
\hline
5 & 0.34576 & 0.29435 & 0.38612 & 0.38612 & 0.34576 & 0.34576 & 0.34576 & 0.38612 & 0.38612 & 0.29435 \\
\hline
6 & 0.32594 & 0.32594 & 0.37403 & 0.37403 & 0.32594 & 0.37403 & 0.37403 & 0.37403 & 0.37403 & 0.29046 \\
\hline
7 & 0.23831 & 0.32682 & 0.32682 & 0.32682 & 0.46188 & 0.46188 & 0.46188 & 0.32353 & 0.32353 & 0.32353 \\
\hline
8 & 0.43962 & 0.43962 & 0.28239 & 0.31136 & 0.4002 & 0.43962 & 0.26125 & 0.43962 & 0.26125 & 0.31136 \\
\hline
9 & 0.3759 & 0.57085 & 0.40988 & 0.25097 & 0.57085 & 0.40988 & 0.25097 & 0.21629 & 0.4198 & 0.24134 \\
\hline
10 & 0.23846 & 0.46884 & 0.46884 & 0.44758 & 0.33438 & 0.33438 & 0.31092 & 0.30199 & 0.3357 & 0.3357 \\
\hline
11 & 0.22934 & 0.43678 & 0.29698 & 0.27807 & 0.5829 & 0.44276 & 0.43678 & 0.44276 & 0.22934 & 0.29698 \\
\hline
12 & 0.30242 & 0.26522 & 0.59962 & 0.43036 & 0.26522 & 0.59962 & 0.43036 & 0.41348 & 0.22841 & 0.21196 \\
\hline
  \end{array}
  \]

  \medskip
  \caption{Twelve non-isometric homogeneous Einstein metrics $\g = -\sum_{\alpha\in \Phi^+} x_{\alpha} B|_{\mathfrak m_\alpha}$ on ${\sf SU}(5)/\mathsf T^4$. Row 1 is the (rescaled) normal homogeneous metric; row 2 is the K\"ahler-Einstein metric; row 3 is the Arvanitoyeorgos metric; rows 4 - 7 are the metrics $g_1, g_2, g_3, g_4$ recently computed in \cite{guzman2024}. These 12 metrics were also computed in \cite{GM16}.}\label{tab:A4}
\end{table}

\medskip

\begin{table}
\[  
  \begin{array}{|c||c|c|c|c|c|c|c|c|c|c|}
    \hline
    & x_{\varepsilon_1 - \varepsilon_2} & x_{\varepsilon_1 - \varepsilon_3} & x_{\varepsilon_2 - \varepsilon_3} &
    x_{\varepsilon_1 + \varepsilon_2} & x_{\varepsilon_1 + \varepsilon_3} & x_{\varepsilon_2 + \varepsilon_3} &
    x_{\varepsilon_1} & x_{\varepsilon_2} & x_{\varepsilon_3} \\    
    \hline \hline
    1 & 0.2 & 0.4 & 0.2 & 0.8 & 0.6 & 0.4 & 0.5 & 0.3 & 0.1 \\
    \hline
    2 & 0.2851 & 0.46136 & 0.67644 & 0.2851 & 0.46136 & 0.21988 & 0.12544 & 0.29905 & 0.47043 \\
    \hline
    3 & 0.35463 & 0.53427 & 0.35463 & 0.35463 & 0.31966 & 0.35463 & 0.36471 & 0.12338 & 0.36471 \\
    \hline
    4 & 0.41893 & 0.41893 & 0.25482 & 0.41893 & 0.41893 & 0.25482 & 0.11463 & 0.42682 & 0.42682 \\
    \hline
    5 & 0.44551 & 0.3448 & 0.3448 & 0.44551 & 0.3448 & 0.3448 & 0.35542 & 0.35542 & 0.12554 \\
    \hline
  \end{array}
  \]

    \medskip
  \caption{Five non-isometric homogeneous Einstein metrics $\g = -\sum_{\alpha\in \Phi^+} x_{\alpha} B|_{\mathfrak m_\alpha}$ on ${\sf SO}(7)/\mathsf T^3$, also computed in \cite{WLZ18}. Note that row 1 is the K\"ahler-Einstein metric.
   }\label{tab:B3}
\end{table}

\begin{table}
\[  
  \begin{array}{|c||c|c|c|c|c|c|c|c|c|c|}
    \hline
    & x_{\varepsilon_1 - \varepsilon_2} & x_{\varepsilon_1 - \varepsilon_3} & x_{\varepsilon_2 - \varepsilon_3} &
    x_{\varepsilon_1 + \varepsilon_2} & x_{\varepsilon_1 + \varepsilon_3} & x_{\varepsilon_2 + \varepsilon_3} &
    x_{2\varepsilon_1} & x_{2\varepsilon_2} & x_{2\varepsilon_3} \\
    \hline 
    1 & 0.125 & 0.25 & 0.125 & 0.625 & 0.5 & 0.375 & 0.75 & 0.5 & 0.25 \\
    \hline
    2 & 0.44264 & 0.44264 & 0.18385 & 0.18385 & 0.18385 & 0.44264 & 0.42641 & 0.42641 & 0.42641 \\
    \hline
    3 & 0.25198 & 0.41425 & 0.44186 & 0.48093 & 0.1603 & 0.15302 & 0.42319 & 0.48599 & 0.37617 \\
    \hline
    4 & 0.42664 & 0.42664 & 0.25055 & 0.15692 & 0.15692 & 0.4834 & 0.37057 & 0.45745 & 0.45745 \\
    \hline
  \end{array}
  \]
  
    \medskip
  \caption{Four non-isometric homogeneous Einstein metrics $\g = -\sum_{\alpha\in \Phi^+} x_{\alpha} B|_{\mathfrak m_\alpha}$ on ${\sf Sp}(3)/\mathsf T^3$, also computed in \cite{GW}. Note that row 1 is the K\"ahler-Einstein metric.
  }\label{tab:C3}
\end{table}

\medskip

\begin{table}
\[  
  \begin{array}{|c||c|c|c|c|c|c|c|c|c|c|c|c|c|}
    \hline
    & x_{\varepsilon_1 - \varepsilon_2} & x_{\varepsilon_1 - \varepsilon_3} & x_{\varepsilon_1 - \varepsilon_4} &
    x_{\varepsilon_2 - \varepsilon_3} & x_{\varepsilon_2 - \varepsilon_4} & x_{\varepsilon_3 - \varepsilon_4} &    
    x_{\varepsilon_1 + \varepsilon_2} & x_{\varepsilon_1 + \varepsilon_3} & x_{\varepsilon_1 + \varepsilon_4} &
    x_{\varepsilon_2 + \varepsilon_3} & x_{\varepsilon_2 + \varepsilon_4} & x_{\varepsilon_3 + \varepsilon_4} \\
    \hline \hline
    1 & 0.33333 & 0.33333 & 0.33333 & 0.33333 & 0.33333 & 0.33333 & 0.33333 & 0.33333 & 0.33333 & 0.33333 & 0.33333 & 0.33333 \\
    \hline
    2 & 0.16667 & 0.33333 & 0.5 & 0.16667 & 0.33333 & 0.16667 & 0.83333 & 0.66667 & 0.5 & 0.5 & 0.33333 &  0.16667 \\
    \hline
    3 & 0.20833 & 0.41667 & 0.41667 & 0.41667 & 0.41667 & 0.20833 & 0.20833 & 0.41667 & 0.41667 & 0.41667 & 0.41667 & 0.20833 \\
    \hline
    4 & 0.35238 & 0.35238 & 0.49333 & 0.2517 & 0.35238 & 0.35238 & 0.35238 & 0.35238 & 0.2517 & 0.2517 & 0.35238 & 0.35238 \\
    \hline
    5 & 0.5183 & 0.31272 & 0.39651 & 0.39651 & 0.31272 & 0.2257 & 0.26036 & 0.39651 & 0.31272 & 0.31272 & 0.39651 & 0.26036 \\
    \hline
  \end{array}
  \]

    \medskip
  \caption{Five non-isometric Einstein metrics $\g = -\sum_{\alpha\in \Phi^+} x_{\alpha} B|_{\mathfrak m_\alpha}$ on ${\sf SO}(8)/\mathsf T^4$. Note that row 1 is the (rescaled) normal homogeneous metric; row 2 is the K\"ahler-Einstein metric.}\label{tab:D4}
\end{table}
\end{landscape}

\pagebreak[4]
  \global\pdfpageattr\expandafter{\the\pdfpageattr/Rotate 0}

\newcommand{\etalchar}[1]{$^{#1}$}
\providecommand{\bysame}{\leavevmode\hbox to3em{\hrulefill}\thinspace}
\providecommand{\MR}{\relax\ifhmode\unskip\space\fi MR }
\providecommand{\MRhref}[2]{%
  \href{http://www.ams.org/mathscinet-getitem?mr=#1}{#2}
}
\providecommand{\href}[2]{#2}

\end{document}